\documentclass[12pt]{amsart}

\usepackage{amsmath}
\usepackage{amsfonts}
\usepackage{amssymb}
\usepackage{algorithm,algorithmic}

\usepackage{multirow}

\usepackage{hyperref}
\usepackage{cleveref}

\def\gl{\mathfrak{gl}}
\def\g{\mathfrak{g}}
\def\h{\mathfrak{h}}

\def\s{\mathfrak{sl}}
\def\so{\mathfrak{so}}
\def\sp{\mathfrak{sp}}

\def\C{\mathbb{C}}
\def\F{\mathbb{F}}

\def\N{\mathbb{N}}
\def\Z{\mathbb{Z}}

\def\ad{\operatorname{ad}}

\def\Der{\operatorname{Der}}
\def\Diag{\operatorname{Diag}}
\def\dim{\operatorname{dim}}
\def\End{\operatorname{End}}
\def\Hom{\operatorname{Hom}}
\def\Ker{\operatorname{Ker}}
\def\Mat{\operatorname{Mat}}
\def\Id{\operatorname{Id}}

\def\Rad{\operatorname{Rad}}
\def\Tr{\operatorname{Tr}}

\def\NilRad{\operatorname{NilRad}}

\newtheorem{Theorem}{Theorem}[section]
\newtheorem{Lemma}[Theorem]{Lemma}
\newtheorem{Prop}[Theorem]{Proposition}
\newtheorem{Def}[Theorem]{Definition}
\newtheorem{Cor}[Theorem]{Corollary}

\newtheorem{Remark}[Theorem]{Remark}

\newtheorem{Conv}[Theorem]{Convention}

\title[On the non-existence of sympathetic Lie algebras]
{On the non-existence of sympathetic Lie algebras with dimension less than 25}

\author{A. L. Garc\'ia-Pulido$^\dagger$}
\address{$^\dagger$
Department of Computer Science\\
University of Liverpool.
}
\address{$^\dagger$
Division of Computing Science and Mathematics\\ 
Cottrell Building\\
University of Stirling \\
Stirling, FK9 4LA, United Kingdom.}
\email{analucia.garciapulido@stir.ac.uk, ana@theoremsfrom.coffee}

\author{G. Salgado$^\ddag$}
\address{$^\ddag$
Facultad de Ciencias\\ 
Universidad Aut\'onoma de San Luis Potos\'i\\ 
Av.\ Parque Chapultepec 1570\\ 
Priv.\ del Pedregal, CP 78210\\ 
San Luis Potos\'i, SLP, M\'exico.}
\email{gsalgado@fciencias.uaslp.mx, gil.salgado@gmail.com}

\keywords {Sympathetic Lie algebras; Equivariant maps; Inner derivations}

\subjclass{
Primary:
17B05, 
17B30. 
Secondary:
17B20 
17B60  
}

\begin{document}

\begin{abstract}
In this article we investigate the question of the lowest possible dimension that a sympathetic Lie 
algebra $\g$ can attain, when its Levi subalgebra $\g_L$ is simple.
We establish the structure of the nilradical of a perfect Lie algebra $\g$, as a $\g_L$-module, and 
determine the possible Lie algebra structures that one such $\g$ admits. We prove that, as a 
$\g_L$-module, the nilradical must decompose into at least $4$ simple modules.
We explicitly calculate the semisimple derivations of a perfect Lie algebra $\g$ with Levi 
subalgebra $\g_L = \s_2(\C)$ and give necessary conditions for $\g$ to be a sympathetic 
Lie algebra in terms of these semisimple derivations.
We show that there is no sympathetic Lie algebra of dimension lower than $15$, independently of the nilradical's decomposition. If the nilradical has $4$ simple modules, we show that a sympathetic Lie algebra has dimension greater or equal than $25$.
\end{abstract}

\maketitle

\section{Introduction} 

A Lie algebra $\g$ is {\it sympathetic} if its centre is trivial, if $\g$ is perfect and if $\g$ 
has no outer derivations; that is, 
\begin{enumerate}\label{symp_defn}
  \item\label{trivial_centre} $Z(\g)=0$,
  \item\label{perfect} $\g=[\g,\g]$
  \item\label{no_outer_der} $\Der(\g) = \ad(\g)$.
\end{enumerate}

It is well known that any semisimple Lie algebra satisfies properties (\ref{trivial_centre}), 
(\ref{perfect}) and (\ref{no_outer_der}). Sympathetic Lie algebras were introduced to determine if properties (\ref{trivial_centre}), (\ref{perfect}) and (\ref{no_outer_der}) completely 
characterise semisimple Lie algebras.
In \cite{A} the first non semisimple Lie algebra satisfying the above properties appears; it has dimension 35. This was followed by an example in dimension 48 \cite{B2}, 
where the term {\it sympathetic} was first used. The lowest known dimension of a sympathetic 
Lie algebra is 25 and can be found in \cite{B}. This algebra has much smaller dimension that the previous examples due to the careful construction with Levi subalgebra equal to $\s_2(\C)$, which is the lowest dimensional simple Lie algebra.

\medskip 
In \cite{B} and \cite{B2}, Benayadi develops theory describing structural properties of sympathetic Lie algebras.
However, the existing results do not lead to a classification, or even showing existence, 
of these algebras in low dimensions. Indeed, the non-existence of sympathetic Lie algebras 
of dimensions less than 10 is attributed to J.\ Simon in \cite{B} and it remains an open 
problem to determine if 25 is the smallest dimension possible in a sympathetic Lie algebra.
The lack of progress could be due to two factors. Little was known of the classification of 
Lie algebras at that time and the available computational resources and power were very limited.

\medskip
In this paper we give a solution to the problem of dimension minimality of sympathetic 
Lie algebras 
which are not semisimple. To do this, we restrict our attention to the class of algebras 
whose Levi factor is a simple Lie algebra.

\medskip
By taking full advantage of the representation theory of simple Lie algebras, we find 
very precise necessary conditions that a Lie algebra $\g= \g_L\ltimes \Rad(\g)$ must satisfy 
in order to be sympathetic. 
In \Cref{cor:1_2_repr}, \Cref{thm:3_repr}, we show that if $\g$ 
is perfect and $\Rad(\g)$ decomposes into 1, 2 or 3 simple $\g_L$-modules, then $\g$ is not sympathetic.
Then $\Rad(\g)$ must decompose into at least $4$ simple $\g_L$-modules.

\medskip
This condition completely characterises sympathetic Lie algebras where $\Rad(\g)$ 
decomposes into 4 representations. In \Cref{thm:posibles_mult} we show that, up to isomorphism, there are only $6$ possible algebra structures. We are therefore able 
to determine that the smallest possible dimension such a sympathetic Lie algebra can have 
is 25, see \Cref{thm:main_is_sl2}. Moreover, we show there is only one such algebra in dimension 25 and it is the example given in \cite{B}. The proof of the non-existence in dimensions smaller than 10 
follows easily from our results. These results are independent of the choice of simple Levi subalgebra.

\medskip
We now describe the techniques used to prove our main theorems.
We make systematic use of each of the properties that define a sympathetic Lie algebra 
$\g=\g_L\ltimes \Rad(\g)$ to derive necessary conditions on its algebra structure. 
Recall that $\Rad(\g)$ admits a $\g_L$-module structure induced by the Lie bracket and 
therefore can be decomposed into a direct sum of simple $\g_L$-modules.

\medskip

A crucial point of our argument is to use the fact that the product between any 
two simple $\g_L$-modules $V,W$ is completely determined by bilinear 
$\g_L$-equivariant morphisms $V\otimes W\to U$, where $U$ is one of the simple $\g_L$-modules 
appearing in the decomposition of $\Rad(\g)$.
In other words, the multiplication is completely determined by the spaces 
$\Hom_{\g_L}(V \otimes W, U)$, where $V,W,U$ are simple $\g_L$-modules appearing in the 
decomposition of $\Rad(\g)$. 
In addition, the nilpotency of $\Rad(\g)$ enables us to derive a general form of the 
multiplicative structure of $\Rad(\g)$. This process further highlights how, when working 
under the assumption that $\g$ is sympathetic, one would construct such an algebra.

\medskip
Combining the above multiplicative structure and Clebsh-Gordan-type theorems for simple $\g_L$-modules, we obtain explicit forms of the possible multiplication tables in terms of the number of simple $\g_L$-modules in the decomposition of $\Rad(\g)$.
The final step is to verify the (non) existence of sympathetic Lie algebras with this explicit multiplication tables in the case of at most 4 modules.

\medskip
The outline of the paper is as follows. For  completeness, we begin Section 2 with preliminary notions of classical Lie algebras. In Section 3, we prove our main theorem \Cref{thm:main}:  if $\g$ is a sympathetic Lie algebra such that its Levi subalgebra $\g_L$ is simple and its nilradical decomposes into four simple $\g_L$-modules, then $\dim_{\C} (\g) \geq 25$. From this point, we work with Lie algebras which Levi subalgebra is simple. In \Cref{subsec:rad_of_symp}, we show that the nilradical of any sympathetic Lie algebra decomposes into at least four simple modules and determine the possible multiplicative structures for this case \Cref{thm:posibles_mult}. In \Cref{subsec:symp_alg_not_sl2} and \Cref{subsec:symp_alg_sl2} we investigate whether any of these possible algebras are sympathetic Lie algebras. 
For part of this analysis, we designed algorithms that enable us to do explicit calculations to this end. In Section 4, we include these algorithms and show in a concrete example how to put to practice our theoretical results and use our algorithms. An implementation of our algorithms can be found in \url{https://github.com/winsy17/Sympathetic_Lie_Algebras}.

\medskip

\section{Preliminaries}

Throughout this article we only consider finite dimensional vector spaces over $\C$.

\begin{Def}{\rm
  A {\it Lie algebra} $\g$, is a vector space together with a bilinear map $[,]:\g \times \g \to \g$ satisfying:
  \begin{enumerate}
  \item\label{eqn:skewsym} $[x,y]= - [y,x]$ for every $x,y\in\g$,
  \item\label{eqn:Jacobi_id} $[x,[y,z]] + [y,[z,x]] + [z,[x,y]]=0$ for every $x,y,z\in \g$.
  \end{enumerate}
  }
\end{Def}

Equation (\ref{eqn:Jacobi_id}) is known as the {\it Jacobi identity}. Notice that (\ref{eqn:skewsym})
is equivalent to $[x,x]=0$ for every $x\in\g$.

\medskip
One of the typical examples of a Lie algebra is the following. Consider a vector space $V$ and the 
associative algebra of endomorphims of $V$ denoted by $\End_{\C}(V)$. Define the {\it commutator} on 
$\End_{\C}(V)$,  by
$[T,S]:= T\circ S - S\circ T$
for every $T,S \in \End_{\C}(V)$. The resulting Lie algebra is denoted by $\gl(V)$ and called the  
{\it general linear algebra associated to $V$.}

\medskip
  Let $\g$ and $\h$ be two Lie algebras. A Lie algebra {\it homomorphism}, is a linear transformation $\rho\colon \g \to \h$ such that
$\rho([x,y]_{\g}) = [\rho(x), \rho(y)]_{\h}$.
An {\it isomorphism} of Lie algebras is a bijective homomorphism.
A {\it linear representation} of a Lie algebra $\g$ on a vector space $V$ is a Lie algebra homomorphism
$\rho\colon \g \to \gl (V)$.
  
\medskip
Let $\alpha\colon\g \to \gl(U)$ and $\beta\colon\g \to \gl(W)$ be two linear representations of $\g$ on $U$ and $W$, respectively. The space of {\it equivariant homomorphisms} from $U$ to $W$ is 
\[
\Hom_{\g}(U,W):=\{ T\in \Hom_{\C}(U,W) \ \vert \ T\circ \alpha(x) = \beta(x) \circ T, \ x\in\g\}.
\]
We can induce a new representation $\alpha\otimes\beta$ from
$\g$ on $U\otimes W$ defined by
$
(\alpha\otimes\beta) (x) (u\otimes w) := \alpha(x)(u) \otimes w + u \otimes \beta(x) (w)
$
for every $x\in\g$, $u\in U$ and $w\in W$. This representation is the {\it tensor product} of $\alpha$
and $\beta$.

\medskip
Given any Lie algebra $\g$ the Lie bracket induces a natural linear representation 
$\ad: \g \to \gl(\g)$ defined by  $\ad_x(y):= [x,y]$ for any $x,y\in \g$. We call this representation 
the {\it adjoint representation}. Notice that  
$\Ker(\ad) = Z(\g)=\{ x\in \g \mid [x,y]=0, \forall y \in \g \}$.

\medskip
  A {\it derivation} on $\g$ is a linear transformation $D\colon\g \to \g$ which satisfies the Leibniz rule 
  with respect to the Lie bracket,
$
  D([x,y]) = [Dx,y] + [x,Dy],
$
  for every $x,y\in\g$.

\begin{Lemma}
Let $\g$ be a Lie algebra. The following statements are equivalent.
\begin{enumerate}
\item $D$ is a derivation in $\g$,
\item $D([e_i, e_j]) = [De_i, e_j] + [e_i, De_j]$ for every $e_i, e_j \in \beta$, 
\item\label{item} $[D, \ad(e_i)] - \ad(De_i) = 0$ for every $e_i \in \beta$.
\end{enumerate}
Where $\beta=\{ e_1, \ldots, e_n\}$ is a basis of $\g$.
\end{Lemma}

\begin{Remark}
\Cref{item} shows that obtaining a basis for the vector space 
$\Der(\g)=\{ D \in \End(\g) \mid D([x,y])=[Dx,y] + [x, Dy] \}$ is equivalent to finding the solutions of a homogeneous system of linear equations.
\end{Remark}

\medskip

Given a linear representation of $\g$ on $V$ $\rho:\g \to \gl(V)$, we say that a subspace 
$W\subset V$ is {\it invariant} if $\rho(x)W \subset W$ for any $x\in\g$.  The trivial subspaces $\{0\}$
and $V$ are invariant subspaces.
We call a representation $\rho\colon\g \to \gl(V)$  {\it irreducible} if the only invariant 
subspaces of $V$ are the trivial ones. 
We say that the representation is {\it completely  reducible} if every invariant subspace 
$W$ admits an invariant complement $U$ such that
$V=W\oplus U$.
  
\medskip
From now on, we say that $V$ is a $\g$-{\it module} if there exists a linear representation $\rho:\g \to \gl(V)$. 
We will say that the $\g$-module $V$ is {\it simple} if $\rho$ is an irreducible representation and, 
{\it semisimple}, if $\rho$ is completely reducible.

\medskip
The very well known classification of simple Lie algebras over $\C$ is given as follows.

\begin{Theorem}[See \cite{H}, \cite{J}, \cite{O-V}]
If $\g$ is a finite dimensional simple Lie algebra  over $\C$, then $\g$ is isomorphic to one of the following Lie algebras.
\begin{enumerate}
\item $A_n := \s(n+1)$, $n\geq 1$ ,
\item $B_n:=  \so(2n+1)$, $n\geq 2$ ,
\item $C_n:= \sp(2n)$, $n \geq 3$, 
\item $D_n:= \so(2n)$, $n\geq 4$, 
\item $E_6$, $E_7$, $E_8$,
\item $F_4$,  
\item $G_2$.
\end{enumerate}
The corresponding dimensions in each case are given by 
\begin{enumerate}
  \item $\dim(\s(n+1)) = n^2+2n$,
  \item $\dim(\so(2n+1)) = 2n^2+n$,
  \item $\dim(\sp(2n))=2n^2+n$,
  \item $\dim(\so(2n)) = 2n^2-n$,
  \item $\dim(E_6) = 78$,
  \item $\dim(E_7) =133$,
  \item $\dim(E_8) = 248$,
  \item $\dim(F_4) = 52$, 
  \item $\dim(G_2) = 14$.
\end{enumerate} 
\end{Theorem}

The only existing isomorphisms between semisimple Lie algebras are the following.
\[
A_1 \simeq B_1 \simeq C_1 \simeq D_1,
\]
\[
B_2 \simeq C_2, \qquad D_2 \simeq A_1 \oplus A_2, \qquad A_3 \simeq D_3
\]

\begin{Theorem}[Weyl]\label{thm:Weyl}
Let $\g$ be a semisimple Lie algebra. Then,
 any finitely dimensional linear representation of $\g$ is completely reducible.
 \end{Theorem}

 \begin{Theorem}[Schur's Lemma]\label{thm:Schur}
 Let $\g$ be a Lie algebra and let
 \[\rho_V\colon \g \to \gl(V) \qquad\text{and}\qquad \rho_W\colon \g \to \gl(W)
 \] 
 be two irreducible representations. Let $T \in \Hom_{\C}(V,W)$ be such that $T\circ \rho_V = \rho_W \circ T$.
\begin{enumerate}
\item If $\rho_V \not\sim \rho_W$ then $T=0$,
\item If $\rho_V = \rho_W$ then $T= \lambda \Id_{V}$.
\end{enumerate}
\end{Theorem}

\section{Main Theorem}

We divide this section in two parts. In the first part we assume that 
$\g=\g_L\ltimes \Rad(\g)$  is a Lie algebra such that its Levi subalgebra $\g_L$ is simple and its radical $\Rad(\g)$ is nilpotent. We show that, if $\g$ is sympathetic, the decomposition of $\Rad(\g)$ into simple $\g_L$-modules must contain at least four non-trivial modules.

\medskip
In the second part, we fix a simple Lie algebra $\g_L$ and determine all the possible 
multiplicative structures that a sympathetic Lie algebra $\g$ can have assuming that $\Rad(\g)$ 
decomposes into four simple $\g_L$-modules. These possible Lie brackets significantly reduce 
the candidate sympathetic Lie algebras. We verify if any of the remaining candidates is a sympathetic Lie algebra. 

\medskip
The main result of this work is the following.
\begin{Theorem}[Main Theorem]\label{thm:main}
  Let $\g$ be a sympathetic Lie algebra with simple Levi subalgebra $\g_L$, then $\dim_{\C} \g \geq 25$.
  \end{Theorem}
  
\medskip
From now on, we will only consider Lie algebras $\g$ with a simple Levi subalgebra $\g_L$ and we will denote the radical $\Rad(\g)$ by $\h$. 
By \Cref{thm:Weyl}, $\h$ can be decomposed, with respect to the $\g_L$-action on $\h$ induced by the Lie bracket, into a sum of simple $\g_L$-modules: 
\[
  \h = V_{n_1} \oplus \cdots \oplus V_{n_k}.
\]
Denote by $s(\g) := k$ to the number of simple $\g_L$-modules in the decomposition of $\h$.
Throughout this work we denote by $V_{n_i}$ to the irreducible representation of $\g_L$ 
with highest weight $n_i$. 
\medskip

Now assume that $\g$ is perfect. It is well known that if $\g$ is perfect then $\Rad(\g)$ is nilpotent (see \cite{J} (Chapter 3, Section 9 Corollary 2)) and therefore $\Rad(\g) = \NilRad(\g)$. Since $\h$ is a nilpotent subalgebra, $0 \neq Z(\h)\subset \h$ and $Z(\h)$ inherits a decomposition as a $\g_L$-submodule of $\h$:
\[ 
Z(\h)= V_{z_1} \oplus \cdots \oplus V_{z_r},
\]
where $\{z_1, \ldots, z_r\} \subseteq \{ n_1, \ldots, n_k\}$.

\begin{Prop}\label{prop:Levi_bracket_symp}
  Let $\g = \g_L \ltimes \h$ be a Lie algebra with simple Levi subalgebra $\g_L$. Given the decomposition $\h = V_{n_1} \oplus \cdots \oplus V_{n_k}$ into simple $\g_L$-modules, then
  \begin{enumerate}
\item\label{item:bracket_simple_irrep_0} $\dim(V_{n_i}) = 1$ if, and only if, $[\g_L,V_{n_i}] = 0$,
\item\label{item:bracket_simple_irrep} if $\dim(V_{n_i}) > 1$, then $[\g_L,V_{n_i}] = V_{n_i}$,
\item\label{item:bracket_two_irreps} if $[V_{n_i},V_{n_j}] \neq 0$, then there exists a $\g_L$-module 
$V_{n_p}$, $1\leq p \leq k$, such that 
$V_{n_p} \subset  V_{n_i} \otimes V_{n_j} $; and a unique $\g_L$-equivariant morphism 
$\Gamma_{i,j}^p\colon V_{n_i} \otimes V_{n_j} \to V_{n_p}$ such that 
$\pi_p ([V_{n_i}, V_{n_j}]) = \Gamma_{i,j}^p(V_{n_i},V_{n_j})$, where $\pi_p \colon V_{n_i} \otimes V_{n_j} \to V_{n_p}$ 
is the natural projection to $V_{n_p}$. If $i=j$, then $V_{n_p} \subset  \Lambda^2(V_{n_i}) $ and 
$\Gamma_{i,i}^p$ is skew-symmetric.

\end{enumerate} 
\end{Prop}

\begin{proof}
The proof of \Cref{item:bracket_simple_irrep_0} and \Cref{item:bracket_simple_irrep} is straightforward.

\medskip
 The fact that any triple $x\in V_{n_i}$, $y\in V_{n_j}$, $z\in \g_L$ satisfies the Jacobi identity is equivalent to the existence of a $\g_L$-equivariant morphism $[,]\colon V_{n_i}\otimes V_{n_j} \to \h$. 
As a $\g_L$-module, there is a decomposition 
\[
V_{n_i}\otimes V_{n_j} = V_{m_1}\oplus \cdots \oplus V_{m_l}.
\] By Schur's lemma, $\{m_1,\ldots,m_l\}\cap \{n_1,\ldots,n_k \} =\emptyset$ implies $[V_{n_i},V_{n_j}]=0$.
\end{proof}
\medskip
\begin{Cor}
  Let $\g = \g_L \ltimes \h$ be a Lie algebra with simple Levi subalgebra $\g_L$. Given the decomposition $\h = V_{n_1} \oplus \cdots \oplus V_{n_k}$ into simple $\g_L$-modules, if $[V_{n_i},V_{n_j}] \neq 0$, 
then there exist simple $\g_L$-modules $V_{m_1},\ldots ,V_{m_r}\subset V_{n_i}\otimes V_{n_j}$ and unique 
$\g_L$-equivariant morphisms $\Gamma_{i,j}^s :V_{n_i} \otimes V_{n_j} \to V_{m_s}$ such that 
    \[
      [V_{n_i},V_{n_j}] = V_{m_1}\oplus\cdots \oplus V_{m_r}, 
      \quad\text{and}\quad
      [x,y] = \sum\limits_{s=1}^r \Gamma_{i,j}^s (x,y),
      \] 
    for every $x\in V_{n_i}$, $y\in V_{n_j}$. 
In the case when $i = j$, $V_{m_1},\ldots ,V_{m_r}\subset \Lambda^2(V_{n_i})$ and we have a skew-symmetric $\g_L$-equivariant morphism $\Gamma_{i}^s :V_{n_i}\otimes V_{n_i} \to V_{m_s}$.
\end{Cor}

\begin{Remark}\label{rmk:equiv_morphisms}
  Given $\g = \g_L\ltimes \h$, the Lie algebra structure of $\h$ is uniquely determined by a 
  collection of $\g_L$-equivariant morphisms, when considering $\h$ as a $\g_L$-module. 
  Conversely, in order to {\it construct} a Lie algebra $\h$ from a $\g_L$-module, 
  one first needs to determine the possible equivariant morphisms which define an algebra structure on $\h$. In fact, any scalar multiple of such morphisms defines a multiplicative structure 
  on $\h$. Therefore, $\h$ is a Lie algebra when there exist an appropriate {\it choice} of multiples for which the Jacobi identity holds.
  \end{Remark}

\begin{Prop}\label{prop:der_semisimples}
Let $\g=\g_L \ltimes \h$ be a finite dimensional sympathetic Lie algebra with simple Levi subalgebra $\g_L$. 
Suppose $D\colon\g\to\g$ is a derivation such that $D\vert_{\g_L} = 0$ and 
$D\vert_{V_{n_i}} = \lambda_{n_i} \Id_{V_{n_i}}$ for any simple $\g_L$-module in the decomposition 
of $\h$. Then $D$ is an inner derivation of $\g$ if, and only if, $D\vert_{\h} = 0$.
\end{Prop}

\begin{proof}
Let $D$ be an inner derivation. Write $D$ as the sum of its semisimple and nilpotent parts $D = \ad_{\g}(s + h)$ for some $s\in \g_L$ and $h\in \h$. By definition of $D$, for any $y\in \g_L$,
\[
0 = D(y) =  [s+h, y] = [s,y] + [h,y].
\] 
Since $[s,y]\in \g_L$ and $[h,y]\in \h$, this implies $[s,y] = 0$ and $[h,y] = 0$ for every $y\in \g_L$.  
Thus $s\in Z(\g_L)$ and hence $s=0$ because $\g_L$ is a simple Lie algebra. 
Consequently, $D$ is nilpotent and this implies $\lambda_{n_i} = 0$ for every $i$.
The converse is straightforward.

\end{proof}

\begin{Prop}\label{prop:der_semisimples_sl2}
  Let $\g=\s_2(\C) \ltimes \h$ be a finite dimensional Lie algebra and $V_m, V_p, V_k$ simple $\s_2(\C)$-modules in the decomposition of $\h$ such that $[V_m,V_p]_{\g}=V_k$, or
  $[V_p,V_p]_{\g}=V_k$. For any semisimple derivation $D\in \Der(\g)$ there exist $\lambda, \alpha,\beta, \gamma \in \F$ such that
  \[
  \begin{aligned}
  \alpha + \beta - \gamma & = \lambda( m+ p - k), \quad \text{or} \\
  2 \beta - \gamma & = \lambda ( 2p - k). 
  \end{aligned}
  \]
  \end{Prop}

\begin{proof}
Since $\s_2(\C)$ is a subalgebra of $\g$, there exists $\lambda \in \F$ such that
\[
D\vert_{\s_2(\C)} = \lambda \ad_{\s_2(\C)}(H).
\]

Suppose that $V_m =\langle e_i \rangle$, $V_p = \langle f_i \rangle$ and
$V_k = \langle g_i \rangle$, where the basis for $V_m,V_p$ and $V_k$ are as in \Cref{prop:explicit_irreps} Since $D$ is a semisimple derivation, we assume without loss of generality that $D$ is of the form
\[
D(e_i) = \alpha_i e_i, \quad D(f_i) = \beta_i f_i, \quad D(g_i) = \gamma_i g_i.
\]  

Given any element $e_i \in V_m$, we have  
\[
D([E,e_i]) = \alpha_{i-1} [E, e_i],
\]
and
\[
\begin{aligned}
\null
[DE, e_i] + [E, De_i] & = 2 \lambda [E,e_i] + \alpha_i [E,e_i] \\
&= (2\lambda + \alpha_i) [E, e_i].
\end{aligned}
\]
Using the fact that $D$ is a derivation we obtain
\[
\alpha_{i-1} = 2\lambda + \alpha_i.
\]
Equivalently,
\[
[D\vert_{V_m}] = \Diag( \alpha_0, \alpha_0 - 2\lambda, \cdots, \alpha_0 - m(2\lambda) ).
\]

Analogously,
\[
\begin{aligned}
\null
[D\vert_{V_p}]  & = \Diag( \beta_0, \beta_0 - 2\lambda, \cdots, \beta_0 - p(2\lambda) ) \\
[D\vert_{V_k}] & = \Diag( \gamma_0, \gamma_0 - 2\lambda, \cdots, \gamma_0 - k(2\lambda) ).
\end{aligned}
\]

Set $\alpha=\alpha_0$, $\beta=\beta_0$ and $\gamma = \gamma_0$. Notice that the product 
$[V_m,V_p]_{\g} = V_k$ is completely determined by an $\s_2(\C)$-equivariant morphism which leaves 
weights invariant. In consequence, if
\begin{equation}\label{eqn:bracket_weight_invariant}
  [e_i, f_j] = C_{i,j} g_{l(i,j)}, \quad C_{i,j}\in \F,
\end{equation} 
then $e_i\otimes f_j \in V_m \otimes V_p$ and $g_{l(i,j)} \in V_k$ have equal weights for every $i,j$. 
A direct calculation shows that  
$l(i,j) = i + j - \frac{1}{2} (m + p - k)$. Applying $D$ to both sides of \Cref{eqn:bracket_weight_invariant} and using the previous expression for $l(i,j)$, we can conclude that 
\[
\alpha + \beta - \gamma = \lambda ( m + p - k).
\]

Similarly, when $m=p$ one obtains
\[
2 \beta - \gamma = \lambda( 2p -k).
\]

\end{proof}

\begin{Remark}
Let $\g=\g_L \ltimes \h$ be a Lie algebra with simple Levi subalgebra $\g_L$. Recall that $\g$ is a sympathetic Lie algebra if $\g$ does not admit any outer derivations. We use \Cref{prop:der_semisimples} and \Cref{prop:der_semisimples_sl2} as criteria to show the existence of outer derivations for certain Lie algebras $\g$ as follows.
  When $\g_L=\s_2(\C)$, \Cref{prop:der_semisimples_sl2} enable us to describe semisimple derivations 
  in terms of a system of homogenous linear equations, which we use to construct outer derivations. 
  When $\g_L\neq \s_2(\C)$, we construct outer derivations of the form described in \Cref{prop:der_semisimples}. 
\end{Remark}

\begin{Remark}
Solving the system from \Cref{prop:der_semisimples_sl2} requires solving a homogenous system of at most $k$ equations with $k+1$ unknowns, where $s(\g)=k$. 
This has a huge computational advantage to calculating directly $\Der(\g)$, which requires solving a system of $(\dim_{\C}\g)^3$ equations, for $\dim_{\C}\g =  n_1 + \cdots + n_k + k + 3$.
\end{Remark}

\begin{Prop}\label{prop:derived_ideal_in_Z}
  Let $\g=\g_L \ltimes \h$ be a perfect Lie algebra such that its Levi subalgebra $\g_L$ is simple and $[\h,\h] \subset Z(\h)$. Then $\g$ is not sympathetic.
  \end{Prop}

\begin{proof} As $\g_L$-modules, $\h$, $[\h,\h]$ and $Z(\h)$ are completely reducible. 
This implies that there exists a $\g_L$-submodule $V$ such that $\h= V \oplus Z(\h)$. Define $D\colon\g \to \g$ by
  \[
  D\vert_{\g_L} \equiv 0, \quad D\vert_{Z(\h)} = \Id_{Z(\h)}, \quad D\vert_V = \frac{1}{2} \Id_V.
  \] 
  A straightforward calculation shows that $D$ is a derivation and by \Cref{prop:der_semisimples}, $D$ is an outer derivation. 
  \end{proof}

\subsection{The radical of a sympathetic Lie algebra}\label{subsec:rad_of_symp}

\subsubsection{Case $s(\g)=1,2$} 
\begin{Cor}\label{cor:1_2_repr}
Let $\g=\g_L \ltimes\h$ be a perfect Lie algebra with simple Levi subalgebra $\g_L$ is simple.  If $s(\g)=1, 2$, then $\g$ is not sympathetic.
\end{Cor}

\begin{proof}
If $s(\g)=1, 2$, then $[\h,\h]\subset Z(\h)$ and \Cref{prop:derived_ideal_in_Z} implies that $\g$ fails to be sympathetic.
\end{proof}

\subsubsection{Case $s(\g)=3$}\label{sec:3_repr}

\begin{Conv}
  From now on, we assume that if $[V_{n_i},V_{n_j}] \neq 0$, there exists a unique $p$, $1 \leq p \leq k$, such that $[V_{n_i},V_{n_j}] = V_{n_p}$. 
  Under this convention we will prove that if $s(\g) \leq 4$, then $Z(\h)$ is a simple $\g_L$-module. 
  For simplicity, we assume that $Z(\h)$ is a simple $\g_L$-module whenever $s(\g) \geq 5$. 
\end{Conv}

In this subsection we show the following 
\begin{Theorem}\label{thm:3_repr}
  Suppose that $\g=\g_L \ltimes\h$ is a Lie algebra, with simple Levi subalgebra $\g_L$, satisfying $[\g,\g]=\g$, $Z(\g) = 0$ and $s(\g)=3$. Then $\g$ is not sympathetic.
\end{Theorem}

\begin{proof}
We determine the {\it multiplication table} for $\h$ using that $\h$ is nilpotent. 
Suppose that $\h= V_m \oplus V_p \oplus V_k$. 
Since $\h$ is a nilpotent subalgebra, 
$Z(\h) = V_m \oplus V_p \oplus V_k$ or $Z(\h) = V_p \oplus V_k$ imply $[\h,\h] \subset Z(\h)$ 
and, by \Cref{prop:derived_ideal_in_Z}, the conclusion of \Cref{thm:3_repr} holds. 

\medskip
Assume that $Z(\h) = V_k$ and notice that $\h^2 = [\h,\h]\neq \h$. Now, 
by \Cref{prop:derived_ideal_in_Z} if $\h^2=Z(\h)=V_k$ then $\g$ is not sympathetic. 
As a result, $\h^2$ must be equal to exactly two 
 simple $\g_L$-modules
\[
\h^2 =[\h,\h] = V_p \oplus V_k, \quad\text{and} \quad
\h^3= V_k = Z(\h).
\]
The first equation implies that $[V_m,V_m]= V_p$ whereas the second implies that $[V_m,V_p]=V_k$. 
From this follows,
\[\begin{aligned}
  0 &= [V_m,[V_m,V_p]] + [V_m,[V_p,V_m]] + [V_p, [V_m,V_m]]\\
    & = [V_m,V_k] + [V_m,V_k] + [V_p, V_p]\\
    & = [V_p, V_p].\\
\end{aligned}
\]

Therefore the only possible multiplication is given by
\[
\begin{tabular}{| c | c | c | c |}
\hline
$[,]$ & $V_m$ & $V_p$ &  $V_k$ \\ \hline
$V_m$ & $V_p$ & $V_k$ &  0 \\ \hline
$V_p$ & $V_k$ & $0$ &  0 \\ \hline
$V_k$ & $0$ & $0$ &  0 \\
\hline
\end{tabular} 
\] Notice that, by \Cref{prop:der_semisimples}, the derivation $D\vert_{\g_L} \equiv 0$,  $D\vert_{V_m} = \frac{1}{2} \Id\vert_{V_m}$, $D\vert_{V_p} = \Id\vert_{V_p}$ and $D\vert_{V_k} = \frac{3}{2} \Id\vert_{V_k}$ is an outer derivation.
\end{proof}

\subsubsection{Case $s(\g)=4$}
A consequence of \Cref{thm:3_repr} is the following. 
\begin{Theorem}\label{thm:symp_at_least_4_repr}
If $\g$ is a sympathetic Lie algebra, then $s(\g) \geq 4$.
\end{Theorem}

Now we are interested in determining all the possible multiplicative structures of nilpotent Lie algebras $\h$ in the case when $s(\g) = 4$, i.e. $\h=V_n \oplus V_m \oplus V_p \oplus V_k$.
\begin{Theorem}\label{thm:posibles_mult}
  Suppose that $\g=\g_L \ltimes\h$ 
  is a perfect Lie algebra with simple Levi subalgebra $\g_L$, $\h= V_n \oplus V_m \oplus V_p \oplus V_k$ and $Z(\h)=V_k$. The algebra structure of $\h$ is completely determined by one of the following
  \begin{enumerate}
  \item $[\h,\h] = Z(\h)$ or,
  \item if $[\h,\h] = V_m \oplus V_p \oplus V_k$, then 
  \[
    \begin{tabular}{| c | c | c | c | c |}
    \hline
    $[,]$ & $V_n$ & $V_m$ & $V_p$ & $V_k$ \\ \hline
    $V_n$ & $V_m$ & $V_p$ & $V_k$ & $0$ \\ \hline
    $V_m$ & $V_p$ & $V_k/0$ & $0$ & $0$ \\ \hline
    $V_p$ & $V_k$ & $0$ & $0$ & $0$ \\ \hline
    $V_k$ & $0$ & $0$ & $0$ & $0$ \\
    \hline
    \end{tabular}
  \] 
  \item or if $[\h,\h] = V_p \oplus V_k$, then 
  \[
    \begin{tabular}{| c | c | c | c | c |}
    \hline
    $[,]$ & $V_n$ & $V_m$ & $V_p$ & $V_k$ \\ \hline
    $V_n$ & $V_p/V_k/0$ & $V_p$ & $V_k$ & $0$ \\ \hline
    $V_m$ & $V_p$ & $V_p/V_k/0$ & $V_k$ & $0$ \\ \hline
    $V_p$ & $V_k$ & $V_k$ & $0$ & $0$ \\ \hline
    $V_k$ & $0$ & $0$ & $0$ & $0$ \\
    \hline
    \end{tabular} 
  \] 
  \[
  \begin{tabular}{| c | c | c | c | c |}
  \hline
  $[,]$ & $V_n$ & $V_m$ & $V_p$ & $V_k$ \\ \hline
  $V_n$ & $V_p/V_k/0$ & $V_p$ & $V_k$ & $0$ \\ \hline
  $V_m$ & $V_p$ & $V_k/0$ & $0$ & $0$ \\ \hline
  $V_p$ & $V_k$ & $0$ & $0$ & $0$ \\ \hline
  $V_k$ & $0$ & $0$ & $0$ & $0$ \\
  \hline
  \end{tabular} 
  \] 
  \[
  \begin{tabular}{| c | c | c | c | c |}
  \hline
  $[,]$ & $V_n$ & $V_m$ & $V_p$ & $V_k$ \\ \hline
  $V_n$ & $V_p$ & $0$ & $V_k$ & $0$ \\ \hline
  $V_m$ & $0$ & $V_k$ & $0$ & $0$ \\ \hline
  $V_p$ & $V_k$ & $0$ & $0$ & $0$ \\ \hline
  $V_k$ & $0$ & $0$ & $0$ & $0$ \\
  \hline
  \end{tabular} 
  \] 
  \[
  \begin{tabular}{| c | c | c | c | c |}
  \hline
  $[,]$ & $V_n$ & $V_m$ & $V_p$ & $V_k$ \\ \hline
  $V_n$ & $V_p$ & $V_k$ & $V_k$ & $0$ \\ \hline
  $V_m$ & $V_k$ & $V_k/0$ & $0$ & $0$ \\ \hline
  $V_p$ & $V_k$ & $0$ & $0$ & $0$ \\ \hline
  $V_k$ & $0$ & $0$ & $0$ & $0$ \\
  \hline
  \end{tabular} 
  \] 
  \end{enumerate}
  \end{Theorem}

\begin{proof}
  Suppose that $[\h,\h]\neq Z(\h)$.
  Since $\g$ is perfect, $\h$ is nilpotent and therefore
  \[
  \h \supsetneq [\h,\h]=\h^2 \supsetneq \h^3 \supsetneq \cdots.
  \] As a consequence, $[\h,\h]$ must be the sum of at most $3$ simple modules and, by hypothesis, $[\h,\h]$ can only contain $2$ or $3$ simple modules.
  
  \medskip
  {\bf (1)} 
  First suppose that $\h^2= V_m \oplus V_p \oplus V_k$ and $\h^3 = V_p \oplus V_k$ y $\h^4= V_k =Z(\h)$.

  \medskip
  From 
  \[
  V_m \subset \h^2, \qquad V_m \not\subset \h^3 = V_p \oplus V_k
  \]
  follows that  $[V_n, V_n] = V_m$.
  \medskip
  Using that
  \[
  [\h,\h^2]= [V_n\oplus V_m \oplus V_p, V_m \oplus V_p] = V_p \oplus V_k
  \] 
  together with 
  \[
  [\h,\h^3]= [V_n\oplus V_m \oplus V_p , V_p] = V_k = Z(\h)
  \] 
  we conclude that
  \[
  [V_p, \h^2 ] =0 \quad \text{and} \quad
  [V_n, V_p ] = V_k
  \]
  \medskip
  Hence the possible multiplicative structure of $\h$ is described in the following table.
  \[
  \begin{tabular}{| c | c | c | c | c |}
  \hline
  $[,]$ & $V_n$ & $V_m$ & $V_p$ & $V_k$ \\ \hline
  $V_n$ & $V_m$ & $V_p/V_k/0$ & $V_k$ & $0$ \\ \hline
  $V_m$ & $V_p/V_k/0$ & $V_p/V_k/0$ & $0$ & $0$ \\ \hline
  $V_p$ & $V_k$ & $0$ & $0$ & $0$ \\ \hline
  $V_k$ & $0$ & $0$ & $0$ & $0$ \\
  \hline
  \end{tabular} 
  \]

  \medskip
  {\bf (1.a)} Assume that $[V_n, V_m]=V_p$. Then
  \[
  0 = [V_m, [V_n,V_m]] + [V_n, [V_m,V_m]] + [V_m,[V_m,V_n]]
   = [V_n, [V_m,V_m]]
  \] and thus 
  $[V_m,V_m] \subseteq V_k$. In this case, the possible multiplication table is given by
  \[
  \begin{tabular}{| c | c | c | c | c |}
  \hline
  $[,]$ & $V_n$ & $V_m$ & $V_p$ & $V_k$ \\ \hline
  $V_n$ & $V_m$ & $V_p$ & $V_k$ & $0$ \\ \hline
  $V_m$ & $V_p$ & $V_k/0$ & $0$ & $0$ \\ \hline
  $V_p$ & $V_k$ & $0$ & $0$ & $0$ \\ \hline
  $V_k$ & $0$ & $0$ & $0$ & $0$ \\
  \hline
  \end{tabular} 
  \]
  
  \medskip
  {\bf (1.b)} By a similar argument to the previous case, if  $[V_n, V_m]=V_k$, then $[V_m,V_m] = V_k/0$. 
  $$
  \begin{tabular}{| c | c | c | c | c |}
  \hline
  $[,]$ & $V_n$ & $V_m$ & $V_p$ & $V_k$ \\ \hline
  $V_n$ & $V_m$ & $V_k$ & $V_k$ & $0$ \\ \hline
  $V_m$ & $V_k$ & $V_k/0$ & $0$ & $0$ \\ \hline
  $V_p$ & $V_k$ & $0$ & $0$ & $0$ \\ \hline
  $V_k$ & $0$ & $0$ & $0$ & $0$ \\
  \hline
  \end{tabular} 
  $$ However, this is impossible as we are assuming  $\h^2$ decomposes into $3$ simple modules.
  
  \medskip
  {\bf (1.c)} Now suppose that $[V_n, V_m]=0$. In this case the possible multiplication table is
  \[
  \begin{tabular}{| c | c | c | c | c |}
  \hline
  $[,]$ & $V_n$ & $V_m$ & $V_p$ & $V_k$ \\ \hline
  $V_n$ & $V_m$ & $0$ & $V_k$ & $0$ \\ \hline
  $V_m$ & $0$ & $V_p/V_k/0$ & $0$ & $0$ \\ \hline
  $V_p$ & $V_k$ & $0$ & $0$ & $0$ \\ \hline
  $V_k$ & $0$ & $0$ & $0$ & $0$ \\
  \hline
  \end{tabular} 
  \]
  \bigskip
  Since we are assuming that $\h^2 = V_m \oplus V_p \oplus V_k$, we have $[V_m,V_m]=V_p$. Therefore
  \[
    \begin{aligned}
  0 &= [V_n, [V_m,V_m]] + [V_m, [V_m, V_n]] + [V_m, [V_n,V_m]]\\
   &= [V_n, [V_m,V_m]] = V_k = Z(\h),
  \end{aligned}
  \] and, in consequence, this case is impossible.

  {\bf (2)} Now suppose that $\h^2=V_p\oplus V_k$. Since $\h$ is nilpotent and $V_p \neq Z(\h)$, we have $ [\h, V_p] \neq V_p,0$ and therefore
  \[
    \h^3 = [\h,\h^2] = [\h,V_p] = V_k. 
  \]
  Hence we may assume $[V_n,V_p] = V_k$ without loss of generality. 

  Since $V_p\subset [\h,\h]$, $V_p = [V,W]$ for some $V,W= V_n,V_m$. 
  \[
  0 = [V_p, [V,W]] + [V, [W, V_p]] + [W, [V_p,V]]  = [V_p, V_p].
  \]
  In summary, the possible multiplication table of $\h$ is of the form,
  \[
  \begin{tabular}{| c | c | c | c | c |}
  \hline
  $[,]$ & $V_n$ & $V_m$ & $V_p$ & $V_k$ \\ \hline
  $V_n$ & $V_p/V_k/0$ & $V_p/V_k/0$ & $V_k$ & $0$ \\ \hline
  $V_m$ & $V_p/V_k/0$ & $V_p/V_k/0$ & $V_k/0$ & $0$ \\ \hline
  $V_p$ & $V_k$ & $V_k/0$ & $0$ & $0$ \\ \hline
  $V_k$ & $0$ & $0$ & $0$ & $0$ \\
  \hline
  \end{tabular} 
  \] 

  {\bf (2.a)} Suppose that $[V_m, V_p] = V_k$. 
  In this case we have that $[V_n,V_m] = V_p$. Otherwise, $[V_n,V_m] = V_k/0$ implies
  \[
    0  = [[V_n,V_m],V_n] +  [[V_m,V_n],V_n] + [[V_n,V_n],V_m]
      = [[V_n,V_n],V_m], 
  \] and 
  \[
    0  = [[V_m,V_n],V_m] +  [[V_n,V_m],V_m] + [[V_m,V_m],V_n]
      = [[V_m,V_m],V_n].  
  \] Hence $[V_n\oplus V_m,V_n\oplus V_m]\subset V_k$, contradicting the fact $V_p\subset [V_n\oplus V_m,V_n\oplus V_m]$. 

  In this case, the multiplication table is
  \[
  \begin{tabular}{| c | c | c | c | c |}
  \hline
  $[,]$ & $V_n$ & $V_m$ & $V_p$ & $V_k$ \\ \hline
  $V_n$ & $V_p/V_k/0$ & $V_p$ & $V_k$ & $0$ \\ \hline
  $V_m$ & $V_p$ & $V_p/V_k/0$ & $V_k$ & $0$ \\ \hline
  $V_p$ & $V_k$ & $V_k$ & $0$ & $0$ \\ \hline
  $V_k$ & $0$ & $0$ & $0$ & $0$ \\
  \hline
  \end{tabular} 
  \] 

  {\bf (2.b)} Now suppose that $[V_m, V_p] = 0$. 
  Then
  \[
  \begin{aligned}
    0 & = [[V_m,V_n],V_m] +  [[V_n,V_m],V_m] + [[V_m,V_m],V_n]\\
     & = [[V_m,V_m],V_n].
  \end{aligned}  
  \] This implies that $[V_m,V_m] = V_k/0$. Since $V_p \subset [\h,\h]$, either $V_p = [V_n,V_m]$ or $V_p = [V_n,V_n]$. In the first case, the following possible multiplication table is given by
  \[
  \begin{tabular}{| c | c | c | c | c |}
  \hline
  $[,]$ & $V_n$ & $V_m$ & $V_p$ & $V_k$ \\ \hline
  $V_n$ & $V_p/V_k/0$ & $V_p$ & $V_k$ & $0$ \\ \hline
  $V_m$ & $V_p$ & $V_k/0$ & $0$ & $0$ \\ \hline
  $V_p$ & $V_k$ & $0$ & $0$ & $0$ \\ \hline
  $V_k$ & $0$ & $0$ & $0$ & $0$ \\
  \hline
  \end{tabular} 
  \] 
  If $V_p \neq [V_m,V_n]$, then using the fact that $V_m\neq Z(\h)$, we obtain two possible multiplicative structures 
  \[
  \begin{tabular}{| c | c | c | c | c |}
  \hline
  $[,]$ & $V_n$ & $V_m$ & $V_p$ & $V_k$ \\ \hline
  $V_n$ & $V_p$ & $0$ & $V_k$ & $0$ \\ \hline
  $V_m$ & $0$ & $V_k$ & $0$ & $0$ \\ \hline
  $V_p$ & $V_k$ & $0$ & $0$ & $0$ \\ \hline
  $V_k$ & $0$ & $0$ & $0$ & $0$ \\
  \hline
  \end{tabular} 
  \] 
  \[
  \begin{tabular}{| c | c | c | c | c |}
  \hline
  $[,]$ & $V_n$ & $V_m$ & $V_p$ & $V_k$ \\ \hline
  $V_n$ & $V_p$ & $V_k$ & $V_k$ & $0$ \\ \hline
  $V_m$ & $V_k$ & $V_k/0$ & $0$ & $0$ \\ \hline
  $V_p$ & $V_k$ & $0$ & $0$ & $0$ \\ \hline
  $V_k$ & $0$ & $0$ & $0$ & $0$ \\
  \hline
  \end{tabular} 
  \]   
\end{proof}

\begin{Cor}\label{cor:perfecta_mult_der_ext}
  Let $\g= \g_L \ltimes \h$  be a perfect Lie with simple Levi subalgebra $\g_L$ and such that $\h=V_n \oplus V_m \oplus V_p \oplus V_k$, $Z(\h) =V_k$. If $[\h,\h]= V_m \oplus V_p \oplus V_k$, then $\g$ admits semisimple outer derivations.
  \end{Cor}

  \begin{proof}
    By \Cref{prop:der_semisimples}, any derivation of the form 
    \[\begin{aligned}
      & D\vert_{\g_L}\equiv 0,\quad D\vert_{V_n} = \frac{\lambda}{4}\Id_{V_n}, \quad D\vert_{V_m} = \frac{\lambda}{2}\Id_{V_m}, \\
      &  D\vert_{V_p} = \frac{3\lambda}{4}\Id_{V_p},\quad D\vert_{V_k} = \lambda\Id_{V_k},
    \end{aligned}
      \] with $\lambda\neq 0 $ is an outer derivation.
  \end{proof}  

\begin{Cor}\label{cor:perfecta_mult_der_ext_2}
    Let $\g= \g_L \ltimes \h$  be a perfect Lie 
    such that $\h=V_n \oplus V_m \oplus V_p \oplus V_k$, $Z(\h) =V_k$ and $[\h,\h]= V_p \oplus V_k$. 
    Then $\g$ has semisimple outer
derivations, if the multiplicative structure of $\h$ is isomorphic to one of the following.
    \begin{enumerate}
      \item \[
        \begin{tabular}{| c | c | c | c | c |}
        \hline
        $[,]$ & $V_n$ & $V_m$ & $V_p$ & $V_k$ \\ \hline
        $V_n$ & $V_p/0$ & $V_p$ & $V_k$ & $0$ \\ \hline
        $V_m$ & $V_p$ & $V_p/0$ & $V_k$ & $0$ \\ \hline
        $V_p$ & $V_k$ & $V_k$ & $0$ & $0$ \\ \hline
        $V_k$ & $0$ & $0$ & $0$ & $0$ \\
        \hline
        \end{tabular} 
      \] 
      
      \item \[
      \begin{tabular}{| c | c | c | c | c |}
      \hline
      $[,]$ & $V_n$ & $V_m$ & $V_p$ & $V_k$ \\ \hline
      $V_n$ & $V_p/V_k/0$ & $V_p$ & $V_k$ & $0$ \\ \hline
      $V_m$ & $V_p$ & $0$ & $0$ & $0$ \\ \hline
      $V_p$ & $V_k$ & $0$ & $0$ & $0$ \\ \hline
      $V_k$ & $0$ & $0$ & $0$ & $0$ \\
      \hline
      \end{tabular} 
      \] 
      \item \[
      \begin{tabular}{| c | c | c | c | c |}
      \hline
      $[,]$ & $V_n$ & $V_m$ & $V_p$ & $V_k$ \\ \hline
      $V_n$ & $V_p$ & $0$ & $V_k$ & $0$ \\ \hline
      $V_m$ & $0$ & $V_k$ & $0$ & $0$ \\ \hline
      $V_p$ & $V_k$ & $0$ & $0$ & $0$ \\ \hline
      $V_k$ & $0$ & $0$ & $0$ & $0$ \\
      \hline
      \end{tabular} 
      \] 
      \item \[
      \begin{tabular}{| c | c | c | c | c |}
      \hline
      $[,]$ & $V_n$ & $V_m$ & $V_p$ & $V_k$ \\ \hline
      $V_n$ & $V_p$ & $V_k$ & $V_k$ & $0$ \\ \hline
      $V_m$ & $V_k$ & $0$ & $0$ & $0$ \\ \hline
      $V_p$ & $V_k$ & $0$ & $0$ & $0$ \\ \hline
      $V_k$ & $0$ & $0$ & $0$ & $0$ \\
      \hline
      \end{tabular} 
      \] 
    \end{enumerate}
\end{Cor}

\begin{proof} 
To show the existence of outer derivations, we now define  derivations satisfying \Cref{prop:der_semisimples} in each of the above cases. As in \Cref{prop:der_semisimples}, we will take $ D\vert_{\g_L}\equiv 0$ in each of the cases.
\medskip

For the first case consider
    \[ D\vert_{V_n\oplus V_m} = \frac{\lambda}{2}\Id_{V_n\oplus V_m},
       D\vert_{V_p} = \lambda\Id_{V_p},\quad D\vert_{V_k} = \frac{3\lambda}{2}\Id_{V_k}
      \] with $\lambda\neq 0 $.
In the second case, define $D$ as before if $[V_n,V_n] = V_p$. Otherwise, set
\[ 
  D\vert_{V_m}\equiv 0,\quad D\vert_{V_n\oplus V_p} = \frac{\lambda}{2}\Id_{V_n\oplus V_p},\quad  D\vert_{V_k} = \lambda\Id_{V_k}
\] 
 with $\lambda\neq 0 $.
For the third case define
\[ 
  D\vert_{V_n} = \frac{\lambda}{2}\Id_{V_n},\quad D\vert_{V_m} = \frac{3\lambda}{4}\Id_{V_m},\quad  D\vert_{ V_p} = \lambda\Id_{V_p} \quad  D\vert_{V_k} = \frac{3\lambda}{2}\Id_{V_k}
\] 
 with $\lambda\neq 0 $.
Finally, set
\[ 
  D\vert_{V_n} = \frac{\lambda}{2}\Id_{V_n}, \quad  D\vert_{V_m\oplus V_p} = \lambda\Id_{V_m\oplus V_p} \quad  D\vert_{V_k} = \frac{3\lambda}{2}\Id_{V_k}
\] 
 with $\lambda\neq 0 $ for the fourth case.
\end{proof}

\begin{Prop}\label{prop:symp_mult_tables}
  Suppose that $\g=\g_L \ltimes \h$ is a Lie algebra with simple Levi subalgebra $\g_L$ such that $\h=V_n \oplus V_m \oplus V_p \oplus V_k$, as a $\g_L$-module. If $\g$ is sympathetic, then there exist two $\g_L$-modules, say $V_p$ and $V_k$, such that $[\h,\h] = V_p\oplus V_k$ and $Z(\h) = V_k$. Moreover, up to isomorphism, the multiplicative structure of $\h$ must be given by one of the $6$ following tables. 
\begin{enumerate}
    \item 
    \[
      \begin{tabular}{| c | c | c | c | c |}
      \hline
      $[,]$ & $V_n$ & $V_m$ & $V_p$ & $V_k$ \\ \hline
      $V_n$ & $V_p$ & $V_p$ & $V_k$ & $0$ \\ \hline
      $V_m$ & $V_p$ & $V_k$ & $V_k$ & $0$ \\ \hline
      $V_p$ & $V_k$ & $V_k$ & $0$ & $0$ \\ \hline
      $V_k$ & $0$ & $0$ & $0$ & $0$ \\
      \hline
      \end{tabular} 
    \] 
    \item
    \[
    \begin{tabular}{| c | c | c | c | c |}
    \hline
    $[,]$ & $V_n$ & $V_m$ & $V_p$ & $V_k$ \\ \hline
    $V_n$ & $V_k$ & $V_p$ & $V_k$ & $0$ \\ \hline
    $V_m$ & $V_p$ & $0$ & $V_k$ & $0$ \\ \hline
    $V_p$ & $V_k$ & $V_k$ & $0$ & $0$ \\ \hline
    $V_k$ & $0$ & $0$ & $0$ & $0$ \\
    \hline
    \end{tabular} 
    \] 
    \item
    \[
    \begin{tabular}{| c | c | c | c | c |}
    \hline
    $[,]$ & $V_n$ & $V_m$ & $V_p$ & $V_k$ \\ \hline
    $V_n$ & $V_k$ & $V_p$ & $V_k$ & $0$ \\ \hline
    $V_m$ & $V_p$ & $V_k$ & $V_k$ & $0$ \\ \hline
    $V_p$ & $V_k$ & $V_k$ & $0$ & $0$ \\ \hline
    $V_k$ & $0$ & $0$ & $0$ & $0$ \\
    \hline
    \end{tabular} 
    \] 
    \item
    \[
    \begin{tabular}{| c | c | c | c | c |}
    \hline
    $[,]$ & $V_n$ & $V_m$ & $V_p$ & $V_k$ \\ \hline
    $V_n$ & $V_p$ & $V_p$ & $V_k$ & $0$ \\ \hline
    $V_m$ & $V_p$ & $V_k$ & $0$ & $0$ \\ \hline
    $V_p$ & $V_k$ & $0$ & $0$ & $0$ \\ \hline
    $V_k$ & $0$ & $0$ & $0$ & $0$ \\
    \hline
    \end{tabular} 
    \] 
    \item
    \[
    \begin{tabular}{| c | c | c | c | c |}
    \hline
    $[,]$ & $V_n$ & $V_m$ & $V_p$ & $V_k$ \\ \hline
    $V_n$ & $V_k$ & $V_p$ & $V_k$ & $0$ \\ \hline
    $V_m$ & $V_p$ & $V_k$ & $0$ & $0$ \\ \hline
    $V_p$ & $V_k$ & $0$ & $0$ & $0$ \\ \hline
    $V_k$ & $0$ & $0$ & $0$ & $0$ \\
    \hline
    \end{tabular} 
    \] 
    \item
    \[
    \begin{tabular}{| c | c | c | c | c |}
    \hline
    $[,]$ & $V_n$ & $V_m$ & $V_p$ & $V_k$ \\ \hline
    $V_n$ & $V_p$ & $V_k$ & $V_k$ & $0$ \\ \hline
    $V_m$ & $V_k$ & $V_k$ & $0$ & $0$ \\ \hline
    $V_p$ & $V_k$ & $0$ & $0$ & $0$ \\ \hline
    $V_k$ & $0$ & $0$ & $0$ & $0$ \\
    \hline
    \end{tabular} 
    \] 
    
  \end{enumerate}
\end{Prop}

\begin{proof}
First notice that if $Z(\h) = \h$ or if $Z(\h) = V_m \oplus V_p \oplus V_k$, then $[\h,\h]\subset Z(\h)$ and by \Cref{prop:derived_ideal_in_Z} $\g$ is not sympathetic. If $Z(\h) = V_p \oplus V_k$, then  similar arguments to those in \Cref{thm:3_repr} show that $\g$ is not sympathetic.

\medskip
Since $\g$ is sympathetic, then $\g$ is perfect and therefore the Lie algebra structure of $\h$ is one of the multiplications given in \Cref{thm:posibles_mult}.
\Cref{prop:derived_ideal_in_Z} implies that, either $[\h,\h] = V_p\oplus V_k$, or $[\h,\h] = V_m\oplus V_p\oplus V_k$. Given that $\g$ does not contain any outer derivations, \Cref{cor:perfecta_mult_der_ext} implies that $[\h,\h] = V_p\oplus V_k$. For the same reason, the possible multiplicative structures of $\h$ in \Cref{thm:posibles_mult} (Item 3) must  exclude those listed in \Cref{cor:perfecta_mult_der_ext_2}.
\end{proof}

\subsection{Sympathetic Lie algebras with simple Levi subalgebra $\g_L \neq \s_2(\C)$}\label{subsec:symp_alg_not_sl2}

\begin{Theorem}\label{thm:smallest_symp_sl_2}
Suppose that $\g = \g_L\ltimes\h$ is a sympathetic Lie algebra with simple Levi subalgebra $\g_L$, $\g_L\neq \s_2(\C)$ and $\h$ does not contain any trivial $\g_L$-modules. Then $\dim(\g)> 25$.
\end{Theorem}

\begin{proof}
  In this proof we use several standard results from representation theory, which can be found for example in \cite{Ha}.
  A simple calculation shows that, if $\g_L \neq \s_3(\C),\s_4(\C), \so(5)$, then $\dim(\g) \geq 21$. 
  By \Cref{thm:symp_at_least_4_repr}, the decomposition of $\h$ into simple $\g_L$-modules has at least $4$ modules. 
  By \Cref{prop:Levi_bracket_symp},
  since $\h$ does not contain any trivial modules, we obtain a lower bound $\dim(\h) \geq 8$. Hence, $\dim(\g) \geq 21+8 = 29$, as required.

  \medskip
  Next we prove the statement for the $3$ remaining cases. First suppose that $\g_L = \s_4(\C), \so(5)$. In this case, $\dim(\g_L) = 15$ and any non trivial simple $\g_L$-module has dimension at least $4$. Using \Cref{thm:symp_at_least_4_repr} we get a bound on the dimension of $\g$,
  \[
  \dim(\g) \geq 15 + 4\cdot 4 = 31.  
  \] 

  Finally, suppose that $\g_L = \s_3(\C)$. Let $V_{(a_1,a_2)}$ denote the simple $\s_3(\C)$-module of highest weight $(a_1,a_2)$. It is a standard result that
  \[
\dim(V_{(a_1,a_2)}) = \frac{1}{2}(a_1+1)(a_2+1)(a_1+a_2+2).
  \] 
  
  Notice that $V_{(0,1)}$ and $V_{(1,0)}$ are the non-trivial simple $\g_L$-modules with lowest dimension and $\dim(V_{(1,0)}) = \dim(V_{(0,1)}) = 3$. 
  We show that if the $\g_L$-module decomposition of $\h$ is of the form $\h = V_{(1,0)}^r\oplus V_{(0,1)}^s$, then $\g$ cannot be a sympathetic Lie algebra. Assume $\h = V_{(1,0)}^r\oplus V_{(0,1)}^s$ as $\g_L$-modules, and set $\h_{(1,0)}=V_{(1,0)}^r$ and $\h_{(0,1)}=V_{(0,1)}^s$.

  \medskip
  From the tensor product decompositions 
    \[
    \begin{aligned}
    V_{(1,0)}\otimes V_{(1,0)} &= V_{(2,0)} \oplus V_{(0,1)} \\
    V_{(0,1)}\otimes V_{(0,1)} &= V_{(0,2)} \oplus V_{(1,0)} \\
    V_{(1,0)}\otimes V_{(0,1)} &= V_{(1,1)} \oplus V_{(0,0)},
  \end{aligned}
  \]
  we see that 
  \[
    [\h_{(1,0)},\h_{(1,0)}]\subset \h_{(0,1)},\quad
    [\h_{(1,0)},\h_{(0,1)}] = 0,\quad
    [\h_{(0,1)},\h_{(0,1)}]\subset \h_{(1,0)}.
  \] 
For any two simple modules $A,B\subset \h$, $[A,B] = C\neq 0$ if, and only if, $A,B\subset \h_{(1,0)}$ or $A,B\subset \h_{(0,1)}$. 
Let $A,B\subset \h$ be simple modules satisfying $[A,B] = C\neq 0$ and suppose, without loss of generality, that 
  $A,B\subset \h_{(1,0)}$ and thus $C\in\h_{(0,1)}$. 
 Since $[B,D]= [A,D] = 0$ for any $D\in \h_{(0,1)}$, we have 
  \[
  0 = [[A,B], D] +[[B,D], A] + [[D,A],B]  = [C,D].
  \] Therefore $C\subset Z(\h)$ and, more generally, we have $[\h,\h]\subset Z(\h)$. By \Cref{prop:derived_ideal_in_Z}, $\g$ is not a sympathetic Lie algebra.
  As a consequence of this, $\h$ must contain at least one simple $\g_L$-module with dimension 
strictly greater than $3$. The smallest possible such simple modules are $V_{(1,1)}$, $V_{(2,0)}$ and 
$V_{(0,2)}$ with $\dim(V_{(1,1)})=\dim(V_{(0,2)})=\dim(V_{(2,0)})=6$. Using this and \Cref{thm:symp_at_least_4_repr} 
we obtain a lower bound 
  \[
    \dim(\g) = \dim (\g_L) + \dim(\h) \geq 8 + 3\cdot 3 + 6= 23. 
 \]
 In particular, if $s(\g)\geq 5$ and $\g$ is a sympathetic Lie algebra, then $\dim(\g) > 25$.
Hence a Lie algebra $\g$ with $\dim(\g)\leq 25$ is sympathetic if, and only if, $\h$ decomposes exactly into $4$ simple modules: one of dimension 6, and three of dimension 3.
  We will show that $\g$ cannot have this structure.

\medskip 
We will only prove the case where $V_{(1,1)}$ is contained in the decomposition of 
$\h$, as the other cases are similar. Recall that the algebra structure of $\h$ is one of the 
multiplications determined in \Cref{prop:symp_mult_tables}. We use the following decompositions of tensor products
  \[
    \begin{aligned}
    V_{(1,1)}\otimes V_{(1,1)} &= V_{(2,2)}\oplus V_{(3,0)}\oplus V_{(0,3)}\oplus V_{(0,0)}\\
    V_{(1,1)}\otimes V_{(1,0)} &= V_{(2,1)}\oplus V_{(1,0)}\oplus V_{(0,2)},\\
    V_{(1,1)}\otimes V_{(0,1)} &= V_{(1,2)}\oplus V_{(0,1)}\oplus V_{(2,0)}.\\
    \end{aligned}
  \]
  Notice that
   \[
     V_{(1,0)}\subsetneq V_{(1,1)}\otimes V_{(1,1)},\qquad V_{(0,1)} \subsetneq V_{(1,1)}\otimes V_{(1,1)},
  \] implies $[V_{(1,1)},V_{(1,1)}] = 0$. Therefore, if the multiplicative structure of $\h$ is isomorphic that of tables 1,3-6, then $V_{(1,1)}= V_p,V_k$; and $V_{(1,1)}= V_m,V_p,V_k$ otherwise.
  The fact that 
   \[
    V_{(1,1)}\subsetneq V_{(1,0)}\otimes V_{(1,0)}, \qquad V_{(1,1)}\subsetneq V_{(0,1)}\otimes V_{(0,1)},
  \] implies $V_{(1,1)}\neq [V_{(0,1)},V_{(0,1)}]$ and $V_{(1,1)}\neq [V_{(1,0)},V_{(1,0)}]$. 
  In consequence, $V_{(1,1)}\neq V_k$ and the algebra structure of $\h$ cannot be isomorphic to tables 1,3,4,6.
  In tables 2 and 5, $V_{(1,1)} = V_p = [V_n,V_m]$ implies that $V_n\not\cong  V_m$. If $V_n = V_{(1,0)}$ then $[V_n,V_p] = [V_{(1,0)},V_{(1,1)}] = V_k$ implies $V_k = V_{(1,0)}$. However $V_k = [V_n,V_n]$ but this contradicts $V_{(1,0)}\subsetneq  V_{(1,0)}\otimes V_{(1,0)}$. 
 Similarly, $V_{(1,1)} \neq V_m$.
\end{proof}

\subsection{Sympathetic Lie algebras with Levi subalgebra $\g_L = \s_2(\C)$}\label{subsec:symp_alg_sl2}

For completeness, we begin this subsection with standard results of the representation theory of $\s_2(\C)$.

\subsubsection{Irreducible representations of $\s_2(\C)$}
Let 
\[
  \s_2(\C)=\lbrace A\in \Mat(2\times 2,\C) \ \vert \ \Tr (A)=0\rbrace
\] and let 
\[
\lbrace H = \left( \begin{smallmatrix} 1&0 \\ 0&-1\end{smallmatrix}\right), E=\left( \begin{smallmatrix} 0&1 \\ 0&0
\end{smallmatrix}\right), F=\left( \begin{smallmatrix} 0&0 \\ 1&0 \end{smallmatrix}\right) \rbrace
\] be a basis  
of $\s_2(\C)$. Notice that $\s_2(\C)$ is a {\it simple} Lie algebra equipped with the commutator $[,]$. Moreover, this choice of basis satisfies
\[
[H,E] = 2E, \quad [H,F]= -2F, \quad [E,F]= H.
\]

\medskip
Denote by $V_n$ to the $n+1$ dimensional vector space with basis 
$\{e_0,e_1,\ldots,e_n\}$. Set $e_{-1}:=e_{n+1}:=0$. We now recall classical results, 
which can be found in \cite{Ha}, \cite{H}, \cite{J} or \cite{O-V}.

\begin{Prop}\label{prop:explicit_irreps}
Let $n\in\N$ and $\rho_n: \s_2(\C) \to \gl(V_n)$ be the linear representation given by
\begin{itemize}
\item $\rho_n(H) (e_i) = (n -2i) e_i$, 
\item $\rho_n(F)(e_i) = e_{i+1}$,
\item $\rho_n(E)(e_i) = i(n+1 -i) e_{i-1}$,
\end{itemize}
for every $0 \leq i \leq n$.
Then
\begin{enumerate}
\item $\rho_n$ is an irreducible representation of $\s_2(\C)$,
\item if $\rho:\s_2(\C) \to \gl(V)$ is an irreducible representation of $\s_2(\C)$ with $\dim V \geq 1$, 
then there exists 
$n \in \N$ such that $V=V_n$
\end{enumerate}
\end{Prop}

We will write $V_0 = \C$ and $\rho_0:\s_2(\C) \to \gl(V_0)$ for the trivial representation
$\rho_0(x) = 0$ for every $x\in \s_2(\C)$.

\medskip
  Given any finite dimensional linear representation $\rho:\s_2(\C) \to \gl(V)$ of $\s_2(\C)$ 
  there is a unique decomposition
  \[
  V= V_{n_1} \oplus \cdots \oplus V_{n_k}.
  \]
  with 
  $\dim_{\C}V = (n_1+1)  + \cdots + (n_k+1)$ for some
  $n_1 \leq n_2 \leq \ldots \leq n_k$.

\begin{Theorem}[Clebsh-Gordan]\label{thm:Clebsh-Gordan}
For any $n \geq m$,
\begin{equation}\label{eqn:tensor_decomp}
V_n \otimes V_m \simeq
\bigoplus_{i=0}^m V_{n+m -2i}
\simeq V_{n+m}\oplus V_{n+m -2} \oplus V_{n+m -4} \oplus \cdots \oplus V_{n-m}.
\end{equation}
\end{Theorem}

\begin{Cor}\label{cor:Clebsh-Gordan_sym_ext}
  For $n=m$,  
  \begin{enumerate}
    \item $\Lambda^2 V_n \simeq  \bigoplus_{i=1}^{[(n+1)/2]} V_{(2n+2)-4i} $,
    \item $S^2 V_n \simeq \bigoplus_{j=0}^{[n/2]} V_{2n - 4j}$,
    \item $V_n \otimes V_n \simeq \Lambda^2 V_n \oplus S^2 V_n$.
  \end{enumerate}
  \end{Cor}

\medskip
Given a Lie algebra $\g$ with Levi decomposition  $\g = \s_2(\C) \ltimes \h$ and an $\s_2(\C)$-module decomposition $\h = V_{n_1}\oplus\cdots\oplus V_{n_k}$, we define
  \[ 
    \h^0 := \bigoplus\limits_{\{n_i \mid\ n_i\equiv 0\mod{2}\}} V_{n_i} \qquad
    \h^1 := \bigoplus\limits_{\{n_i \mid\ n_i\equiv 1\mod{2}\}} V_{n_i}.
  \]
We show that the product in $\h$ is $\Z_2$-graded:

\begin{Lemma}\label{lem:Z_2_graded}
Let $\g$ be a Lie algebra, with Levi decomposition given by $\g = \s_2(\C) \ltimes \h$ and $\h = \h^0 \oplus \h^1$. Then $\h^0$ is a subalgebra of $\h$ and $[\h^1,\h^1]\subset \h^0$ and $[\h^1,\h^0] \subset \h^1$.
  In particular, if $[V_{n_i},V_{n_j}] = V_{n_l}$ then $n_i + n_j \equiv n_l\mod{2}$. 
\end{Lemma}
\begin{proof}
  By the \Cref{thm:Clebsh-Gordan}, the decompositions \[
  V_{2n+1}\otimes V_{2m+1} = V_{2(n+m)+2} \oplus \cdots \oplus V_{2(m-n)}
  \] 
  \[
  V_{2n}\otimes V_{2m} = V_{2(n+m)} \oplus \cdots \oplus V_{2(m-n)}
  \]
  only contain simple $\s_2(\C)$-modules with $p\equiv 0 \mod{2}$.
  
\end{proof}

\begin{Lemma} \label{lem:impares}
Let $\g$ be a Lie algebra with Levi decomposition given by $\g = \s_2(\C) \ltimes \h$. If $\h = \h^1$, then the radical $\h$ is abelian.  In particular, $\g$ is not sympathetic.
\end{Lemma}

\begin{proof}
Combining \Cref{lem:Z_2_graded} with Schur's Lemma \ref{thm:Schur} we conclude 
that any bilinear $\s_2(\C)$-equivariant morphism 
$\Gamma\colon V_{2n+1}\otimes V_{2m+1} \to V_{2l+1}$ is trivial.
In particular, the Lie bracket $[\ ,\ ]\colon V_{n}\otimes V_{m} \to \h$ is zero for every 
$n,m\in\{n_1, \ldots , n_k\}$. By \Cref{prop:derived_ideal_in_Z}, $\g$ is not a sympathetic Lie algebra.

\end{proof}

\begin{Cor}\label{cor:even_Z}
Let $\g$ be a Lie algebra with Levi decomposition $\g = \s_2(\C) \ltimes\h$ and $\h=\h^0\oplus\h^1$.
If $\h^0\subset Z(\h)$, then $\g$ is not a sympathetic Lie algebra.
  \end{Cor}

\begin{proof} 
From Clebsh-Gordan's \Cref{thm:Clebsh-Gordan} follows that $[\h,\h] \subset Z(\h)$. By \Cref{prop:derived_ideal_in_Z}, $\g$ is not a sympathetic Lie algebra.

  \end{proof}

\begin{Cor}\label{thm:even_irreps}
Let $\g$ be a sympathetic Lie algebra with Levi decomposition given by $\g= \s_2(\C) \ltimes \h$ and with $\h= V_n \oplus V_m \oplus V_p \oplus V_k$ and $Z(\h)=V_k$. Then $n\equiv m\equiv p\equiv k\equiv 0 \mod{2}$. Furthermore, in each case of \Cref{prop:symp_mult_tables} the following additional conclusions hold.
\begin{enumerate}
  \item $p\equiv k \equiv 2 \mod{4}$,
  \item $k \equiv 2 \mod{4}$,
  \item $p\equiv k \equiv 2 \mod{4}$,
  \item $k \equiv 2 \mod{4}$,
  \item $p\equiv k \equiv 2 \mod{4}$.
\end{enumerate}
\end{Cor}

\begin{proof}
The fact that $n\equiv m\equiv p\equiv k\equiv 0 \mod{2}$ follows directly by applying 
\Cref{lem:Z_2_graded} on each of the possible multiplicative structures of $\h$ 
from \Cref{prop:symp_mult_tables}. 

\medskip
To prove the remaining statements, it is sufficient to observe that if two simple 
$\s_2(\C)$-modules $V_r,V_s$ satisfy $V_r = [V_s,V_s]$ then the skew-symmetry of 
the Lie bracket implies that $V_r\subset \Lambda^2 (V_s)$. Using the decomposition 
from \Cref{thm:Clebsh-Gordan}, we conclude that $r\equiv 2\mod{4}$.

\end{proof}

\begin{Cor}
Let $\g$ be a sympathetic Lie algebra with Levi decomposition given by $\g= \s_2(\C) \ltimes \h$,  such that $\h=V_n \oplus V_m \oplus V_p \oplus V_k$, and $Z(\h) =V_k$. Then $\dim\g = 2N +1$.
\end{Cor}

\medskip
\begin{Prop}\label{thm:possible_cases_sl2}
Let $\g$ be a sympathetic Lie algebra with Levi decomposition given by $\g= \s_2(\C) \ltimes \h$. Suppose that $\h=V_n \oplus V_m \oplus V_p \oplus V_k$, as an $\s_2(\C)$-module, and that $Z(\h) =V_k$, where $n,m,p,k\in\N$. If $\dim(\g)\leq 25$ then one of the following cases must necessarily hold:
  \begin{itemize}
    \item the multiplication is given by 
    one of the tables 1-6 and $(n,m,p,k)$ equals one of
    \begin{gather*}
      (4,4,2,6),(4,6,2,6),(6,4,2,6),(4,4,6,2) ,
       (6,4,6,2),(4,6,6,2),\\(2,2,2,2) ,(2,4,2,2) ,
       (4,2,2,2),(4,4,2,2)
    \end{gather*} 
    \item the multiplication is given by 
    one of the tables 2,3,5 and $(n,m,p,k)$ equals one of
    \begin{gather*}
      (4,4,4,6),(2,2,4,2),(2,4,4,2),(2,6,4,2),
      (4,2,4,2),(4,4,4,2),\\(4,6,4,2),(6,2,4,2),
      (6,4,4,2),(6,6,4,2)
    \end{gather*}
    
    \item the multiplication is given by 
    one of the tables 4,5,6 and $(n,m,p,k)$ equals one of
    \[(4,2,6,2),(4,6,2,2) 
    \]
    \item the multiplication is given by 
    one of the tables 4,5 and $(n,m,p,k)$ equals one of
    \[ (6,2,6,2),(8,2,6,2)\] 
    \item the multiplication is given by 
    table 2 and $(n,m,p,k)$ equals one of
    \[ (4,2,4,6),(6,2,4,6),(4,2,6,6)
    \] 
    \item  the multiplication is given by 
    table 5 and $(n,m,p,k)$ equals one of
    \[(6,2,8,2), (4,8,4,2).
    \]
  \end{itemize}
  
\end{Prop}
\begin{proof}
  Using the parity of $n,m,k,p$ found in \Cref{thm:even_irreps}, we write $n= 2a$, $m= 2b$, $p=2c$ and $k = 4d+2$. Recall that the possible multiplicative structure of $\h$ is one of the tables in \Cref{prop:symp_mult_tables}.
Notice that $\dim(\g)= 2(a+b+c+2d)+9\leq 25$ if and only if $a+b+c+2d\leq 8$.

We only give the proof for $d\geq 1$ as the case $d = 0$ is very similar.
By Clebsh-Gordan theorem, $V_{4d+2}  \subsetneq V_2\otimes V_2$. In each of the above tables we see that 
$V_{4d+2} = [V_{2a},V_{2c}]$ and hence $a \geq 2$ or $c\geq 2$. 
Similarly, in tables 1,3,4,5,6, 
\[ 
  V_{4d+2} = [V_{2b},V_{2b}]
\] and therefore $b\geq 2$. In particular, $a+b+c\geq 5$.
Further, in table 2
\[
  V_{4d+2} = [V_{2a},V_{2a}]\qquad\text{ and }\qquad V_{4d+2} = [V_{2b},V_{2c}]
  \] 
and so $a+b+c\geq 5$ in this case too. 
Since we require $a+b+c+2d\leq 8$ then $d= 1$ and hence 
\begin{equation}\label{eqn:bounds}
5\leq a + b + c \leq 6.
\end{equation}

In tables 1,4,6, $c\equiv 1\mod{2}$ since $[V_{2a},V_{2a}] = V_{2c}$. By similar arguments to before, $c\geq 3$ and $[V_{2a},V_{2a}] = V_{2c}$ imply $a\geq 2$ and $a+b+c\geq 7$. Hence $c=1$ so that \Cref{eqn:bounds} is satisfied. Since $[V_{2a},V_{2c} ] = V_6$, then $a\geq 2$.  Then we have $a,b\geq 2$, $c=1$ and $d=1$. In this case, the only possibilities for $\h$ are 
\begin{center}
\begin{tabular}{| c | c | c | c | c | c |}
  \hline
  \multicolumn{6}{|c|}{Multiplications 1,4,6, $d=1$} \\
  \hline
  $a$ & $b$ & $c$ & $d$ & $\h = V_{2a}\oplus V_{2b}\oplus V_{2c} \oplus V_{4d+2}$& $\dim(\g)$ \\ \hline
   $2$ & $2$ & $1$  &  $1$ & $V_4\oplus V_4 \oplus V_2\oplus V_6$ & $23$ \\ \hline
   $2$ & $3$ & $1$ & $1$  &  $V_4\oplus V_6 \oplus V_2\oplus V_6$ & $25$ \\ \hline
   $3$ &  $2$ & $1$ & $1$ &  $V_6\oplus V_4 \oplus V_2\oplus V_6$ & $25$ \\ \hline
\end{tabular} \end{center}

Using the same arguments than before in tables 3 and 5, we obtain $a,b\geq 2$.
Therefore, the $a,b,c$ that satisfy \Cref{eqn:bounds} are
\begin{center}
\begin{tabular}{| c | c | c | c | c | c |}
  \hline
  \multicolumn{6}{|c|}{Multiplications 3,5, $d=1$} \\
  \hline
  $a$ & $b$ & $c$ & $d$ & $\h = V_{2a}\oplus V_{2b}\oplus V_{2c} \oplus V_{4d+2}$& $\dim(\g)$ \\ \hline
   $2$ & $2$ & $1$  &  $1$ & $V_4\oplus V_4 \oplus V_2\oplus V_6$ & $23$ \\ \hline
   $2$ & $3$ & $1$ & $1$  &  $V_4\oplus V_6 \oplus V_2\oplus V_6$ & $25$ \\ \hline
   $3$ &  $2$ & $1$ & $1$ &  $V_6\oplus V_4 \oplus V_2\oplus V_6$ & $25$ \\ \hline
   $2$ &  $2$ & $2$ & $1$ &  $V_4\oplus V_4 \oplus V_4\oplus V_6$ & $25$ \\ \hline
\end{tabular} \end{center}

For Table 2, we previously showed that $a\geq 2$ and either $b\geq 2$ or $c\geq 2$. In this case, the possible $a,b,c,d$ are

\begin{center}\begin{tabular}{| c | c | c | c | c | c |}
  \hline
  \multicolumn{6}{|c|}{Multiplication 2, $d=1$} \\
  \hline
  $a$ & $b$ & $c$ & $d$ & $\h = V_{2a}\oplus V_{2b}\oplus V_{2c} \oplus V_{4d+2}$& $\dim(\g)$ \\ \hline
   $2$ & $2$ & $1$  &  $1$ & $V_4\oplus V_4 \oplus V_2\oplus V_6$ & $23$ \\ \hline
   $2$ & $3$ & $1$ & $1$  &  $V_4\oplus V_6 \oplus V_2\oplus V_6$ & $25$ \\ \hline
   $3$ &  $2$ & $1$ & $1$ &  $V_6\oplus V_4 \oplus V_2\oplus V_6$ & $25$ \\ \hline
   $2$ &  $2$ & $2$ & $1$ &  $V_4\oplus V_4 \oplus V_4\oplus V_6$ & $25$ \\ \hline
   $2$ &  $1$ & $2$ & $1$ &  $V_4\oplus V_2 \oplus V_4\oplus V_6$ & $23$ \\ \hline
   $3$ &  $1$ & $2$ & $1$ &  $V_6\oplus V_2 \oplus V_4\oplus V_6$ & $25$ \\ \hline
   $2$ &  $1$ & $3$ & $1$ &  $V_4\oplus V_2 \oplus V_6\oplus V_6$ & $25$ \\ \hline
\end{tabular} \end{center}

\end{proof}

\begin{Remark}
Notice that any $\s_2(\C)$-module $\h = V_n\oplus V_m\oplus V_p\oplus V_k$, with any choice of multiplicative structure in \Cref{prop:symp_mult_tables}, is an algebra but not necessarily a Lie algebra.
\end{Remark}

For the analysis of the remaining cases, we use \Cref{Lemma:family_no_Lie} to show that most of the cases from \Cref{thm:possible_cases_sl2} are not Lie algebras.
Indeed, given a fixed value of $(n,m,p,k)$, we consider the associated vector space $\h$ and make a choice of $\g_L$-equivariant morphisms between the simple modules of $\h$ which is compatible with the respective multiplication prescribed in \Cref{thm:possible_cases_sl2}. Recall from \Cref{rmk:equiv_morphisms} that in order to show that $\h$ with a prescribed multiplicative structure is not a Lie algebra, it is necessary to prove that 
after {\it any} non-zero rescaling of these morphisms, $\h$ does not satisfy the Jacobi identity.
 In other words, it is necessary to show that any algebra in the family resulting from the non trivial rescalings does not satisfy the Jacobi identity. The following Lemma gives us a criterion to show no element of the family is a Lie algebra.

\begin{Lemma}\label{Lemma:family_no_Lie}
  Let $\g_L$ be a simple Lie algebra and $\h=V_{n_1}\oplus \cdots\oplus V_{n_k}$ be a $\g_L$-module with $V_{n_j}$ simple. For every $1\leq i\leq j\leq k$, suppose that there is $1\leq p\leq k$ and a $\g_L$-equivariant morphism
  \[
    \Gamma_{i,j}^p :V_{n_i} \otimes V_{n_j} \to V_{n_p},
  \] such that $\Gamma_{i,i}^p$ is skew-symmetric. Consider the resulting algebra $\g = \g_L\ltimes \h$, which product $[,]$ is constructed from the $\Gamma_{i,j}^p$.  
  
  Suppose that 
  \begin{enumerate}
  \item\label{item:homogeneous_Jacobi} there exist $a,b,c\in \g$, such that 
  \[
    [a,[b,c]] = [b,[c, a]] = 0\qquad\text{and}\qquad[c,[a,b]] \neq 0,
  \] 
  \item\label{item:Jacobi_fails_nnm} or, there exist $1\leq i,j\leq k$ and $a,b \in V_{n_i}$ $c\in V_{n_j}$ such that 
  \[
    [a,[b,c]] + [b,[c, a]] \neq 0   \qquad\text{and}\qquad  [c,[a,b]]=0,
  \] 
  \item\label{item:not_subalgebra} or there exist 
  $1\leq i\leq k$ and $a,b,c \in V_{n_i}$ that do not satisfy the Jacobi identity.  
  \end{enumerate} Then any algebra $\tilde{\g}= \g_L\ltimes \h$ obtained from non-zero rescaling $\Gamma_{i,j}^p$ does not satisfy the Jacobi identity.
  \end{Lemma}

  \begin{proof}
  The proof of \Cref{item:homogeneous_Jacobi} is straightforward.
  Let $i,j$ and $a,b,c$ as in \Cref{item:Jacobi_fails_nnm}. The equation
  \[
  0 = [c,[a,b]] = \Gamma^q_{p,m}(c\otimes\Gamma^p_{n,n}(a,b))
  \] is homogeneous and therefore satisfied by any rescaling of $\Gamma^q_{p,m},\Gamma^p_{n,n}$. Similarly, by definition of $[\ ,\ ]$ we have that 
  \[  \begin{aligned}
   \null [a,[b,c]] + [b,[c, a]] &= \Gamma^s_{n,r}(a\otimes\Gamma^r_{n,m}(b,c)) -  \Gamma^s_{n,r}(b\otimes\Gamma^r_{n,m}(a,c))  \\
    &= \Gamma^s_{n,r}\left(a\otimes\Gamma^r_{n,m}(b,c) - b\otimes\Gamma^r_{n,m}(a,c)\right),
  \end{aligned}
    \] is also homogeneous with respect to $\Gamma^s_{n,r},\Gamma^r_{n,m}$ and therefore non-zero when rescaling.

  The proof of \Cref{item:not_subalgebra} is very similar.
  \end{proof}

Applying the previous lemma to each of the cases listed in \Cref{thm:possible_cases_sl2} we obtain the following.
\begin{Cor} 
Suppose $\g = \s_2(\C)\ltimes \h$ is an algebra where $\h = V_n\oplus V_m\oplus V_p\oplus V_k$ is one of the algebras listed in the conclusion of \Cref{thm:possible_cases_sl2}. If $\g$ is a Lie algebra then one of the following cases must necessarily hold
\begin{enumerate}
  \item the algebra structure of $\h$ is given by table 5 and $(n,m,p,k)=(2, 6, 4, 2),$ or,
  \item the algebra structure of $\h$ is given by table 6 and $(n,m,k,p)=(6, 4, 6, 2),(2, 2, 2, 2)$ or $(2, 4, 2, 2).$ 
\end{enumerate}
\end{Cor}

\begin{Prop}
  Suppose $\g = \s_2(\C)\ltimes \h$ is an algebra where $\h = V_n\oplus V_m\oplus V_p\oplus V_k$ is one of the algebras listed in the conclusion of \Cref{thm:possible_cases_sl2}. The following algebras admit outer derivations.
  \begin{enumerate}
    \item the multiplication of $\h$ is given by table 6 and $(n,m,p,k) = (2, 2, 2, 2),(2, 4, 2, 2),$
    \item the  multiplication of $\h$ is given by table 5 and $(n,m,p,k)=(2, 6, 4, 2)$.
    \end{enumerate}
\end{Prop}

\begin{Theorem}\label{thm:main_is_sl2}
Let $\g=\g_L\ltimes \h$ be a sympathetic, non semisimple Lie algebra, with simple Levi subalgebra $\g_L$ and $\h = V_n \oplus V_m \oplus V_p \oplus V_k$, as a $\g_L$-module. Then $dim(\g) \geq 25$. Moreover, $\dim(\g) = 25$ if and only if $(n,m,p,k) = (6,4,6,2)$ and $\g_L = \s_2(\C)$.
\end{Theorem}

\begin{Theorem}
 Let $\g=\g_L\ltimes \h$ be a sympathetic, non semisimple Lie algebra, with simple Levi subalgebra $\g_L$. Then $\dim(\g) \geq 15$.
\end{Theorem}

\begin{proof}
The case $\g_L\neq \s_2(\C)$ is an immediate consequence of \Cref{thm:smallest_symp_sl_2}.

Suppose $\g_L=\s_2(\C)$ and set $\h = V_{n_1}\oplus\ldots\oplus V_{n_k}$. By \Cref{thm:symp_at_least_4_repr}, $s(\g)=k\geq 4$. \Cref{thm:main_is_sl2} gives us the  bound for the case $s(\g) =4$. Assume that $s(\g) = k \geq 5$. Since $\g$ is sympathetic, \Cref{lem:impares} implies that $n_i > 1$ for at least one $i = 1,\ldots,k$. In fact, there exist at least two $i,j$ such that $n_i,n_j>1$. Otherwise, using the notation from \Cref{lem:Z_2_graded}, either $\h^0 = Z(\h)$ or $\g$ admits an outer derivation similar to those we constructed before. In consequence, we get the required bound
\[
\dim(\g)  = 8 + \sum_{i=1}^k{n_i}\geq 15.  
\]  
\end{proof}

\medskip
\begin{Theorem}
Let $\g$ be a perfect Lie algebra with Levi decomposition $\g= \s_2(\C) \ltimes \h$ and $Z(\g)=0$. Suppose that $\h=V_n \oplus V_m \oplus V_p \oplus V_k$, $Z(\h) =V_k$, as $\s_2(\C)$-modules, where $n,m,p,k\in\N$. If $\dim(\g) < 25$ then $\g$ is not a sympathetic Lie algebra.
\end{Theorem}

\section{Appendix: Algorithmic Construction and Verification of Sympathetic Lie Algebras}

In this section we include the algorithms that we use to construct and verify particular cases of Lie algebras with Levi subalgebra $\g_L = \s_2(\C)$. Our Python implementation of these algorithms can be found in \url{https://github.com/winsy17/Sympathetic_Lie_Algebras}.

\subsection{Algorithm to construct simple $\s_2(\C)$-submodules of $V_n\otimes V_m$}\label{sec:codigo_generar_V_k}

This algorithm takes integers $n \geq m\geq 0$ and $k$ as input, and returns a basis of $V_k$ as a  simple submodule of $V_n\otimes V_m$. We require $k = n+m -2r$ for some $0\leq r\leq m$, in order to guarantee that $V_k$ is a simple submodule of $V_n \otimes V_m$ (see \Cref{thm:Clebsh-Gordan}). 

\medskip
We use the notation $e_i\odot f_j$ to indicate an  element of the basis of the respective vector space $V_n\otimes V_m$, $S^2(V_n)$ or $\Lambda^2(V_n)$. In the case of $S^2(V_n)$, when $i\neq j$, we use $e_i\odot e_j = e_i\otimes e_j + e_j \otimes e_i$ and $e_i\odot e_i = e_i\otimes e_i$. For $\Lambda^2(V_n)$ we use the standard basis $e_i\odot e_j = e_i\otimes e_j - e_j \otimes e_i$, where $1\leq i<j\leq n$.

\begin{algorithm}
  \caption{Finding a basis of $V_k$ as a submodule of $V_n\otimes V_m$, $S^2(V_n)$ or $\Lambda^2(V_n)$.}
  \begin{algorithmic}[1]  
  \STATE {\bfseries Input}: $n\geq m$, $k$, corresponding to $V_n,V_m, V_k$;\\
  \STATE {\bfseries Output:} Basis of $V_k$ from decompositions \Cref{thm:Clebsh-Gordan}, \Cref{cor:Clebsh-Gordan_sym_ext}
  \\
  \STATE $W_{k} := \{e_i\odot f_j : (\rho_n\otimes\rho_m)(H)(e_i\odot f_j) = k \cdot (e_i\odot f_j)\}$
  \STATE $W_{k+2} := \{e_i\odot f_j : (\rho_n\otimes\rho_m)(H)(e_i\odot f_j) = (k + 2) \cdot (e_i\odot f_j)\}$ 
  \IF{$W_{k+2} = \emptyset$}
    \STATE{$u_k \leftarrow W_{k}$}
  \ELSE 
    \STATE{$u_k \leftarrow \operatorname{Ker}(\rho_n\otimes\rho_m)(E)|_{W_{k}}$}
  \ENDIF 
  \STATE{$\mathcal{B} := \{(\rho_n\otimes\rho_m)^j(F)(u_k): j = 0,\ldots,k\}$}
  \RETURN $\mathcal{B}$
  \end{algorithmic}
  \end{algorithm}

\subsection{Algorithm to construct $\s_2(\C)$-equivariant morphisms}

This algorithm receives integers $n\geq m$ and $k$ as input. If $V_k$ is a simple submodule of $V_n\otimes V_m$, then the algorithm returns a matrix corresponding to a non trivial equivariant morphism $\Gamma_{n,m}^k\colon V_n\otimes V_m \to V_k$. Otherwise, it returns the zero matrix. 

\medskip
We use $[\mathcal{A}]$ to denote the matrix which rows are the elements of an ordered basis $\mathcal{A}$; $\mathbb{O}_{r,s}$ the zero matrix of dimension $r\times s$. In steps \ref{alg_step:image_start}-\ref{alg_step:image_end} we construct the matrix $I$ of dimension $(n+1)(m+1)\times (k+1)$. With the exception of an identity submatrix of dimension $(k+1)\times (k+1)$, the matrix $I$ only contains zeros. These zeros correspond to the morphisms
\[
  V_r\hookrightarrow V_n\otimes V_m \xrightarrow{0} V_k
\] when $r\neq k$ and 
\[
  V_k \hookrightarrow V_n\otimes V_m \xrightarrow{\Id} V_k.
  \]

\begin{algorithm}
  \caption{Constructing an $\s_2(\C)$-equivariant morphism $\Gamma \colon V_n\otimes V_m\to V_k$.}
  \begin{algorithmic}[1]  
  \STATE {\bfseries Input}: $n\geq m$, $k$, corresponding to $V_n,V_m, V_k$
  \STATE {\bfseries Output:} Matrix associated to $\Gamma$ in terms of the standard basis of $V_n\otimes V_m$.
  \IF{$V_k$ is not a simple submodule of $V_n\otimes V_m$}
    \RETURN{$\mathbb{O}_{(k+1),(n+1)(m+1)}$}
  \ENDIF
  \IF{$n=m$ and $2n-k\equiv 0 \operatorname{mod} 4$}\label{alg_step:symmetric}
      \STATE{$W = S^2(V_n)$}
  \ELSIF{$n=m$ and $2n + 2 -k\equiv 0 \operatorname{mod} 4$}\label{alg_step:antisymmetric}
      \STATE{$W = \Lambda^2(V_n)$}
  \ELSE
    \STATE{$W = V_n\otimes V_m$}\label{alg_step:bilinear}
  \ENDIF
  \STATE{$\mathcal{B}=\emptyset$, $\text{I} = []$}\label{alg_step:image_start}
  \FORALL{$V_r$ simple submodules of $W$}
    \STATE{$\mathcal{B} = \mathcal{B} \cup \{u_r,\rho(F)(u_r),\ldots,\rho^r(F)(u_r)\}$}
    \IF{$r\neq k$}
      \STATE{$\text{I} = \text{I} + [\mathbb{O}_{r+1,k+1}]$}
    \ELSE
      \STATE{$\text{I} = \text{I} + [\Id_{k+1,k+1}]$}
    \ENDIF
  \ENDFOR\label{alg_step:image_end}
  \STATE{$\Gamma = [\mathcal{B}]^{-1}\cdot \text{I}$}\label{alg_step:gamma_inverse}
  \IF{$W = V_n\otimes V_m$}
    \RETURN{$\Gamma$}
  \ELSIF{$W = S^2(V_n)$}
    \RETURN{Extension of $\Gamma$ to $V_n\otimes V_m$ as a symmetric function.}
  \ELSE
    \RETURN{Extension of $\Gamma$ to $V_n\otimes V_m$ as an antisymmetric function.}  
  \ENDIF
  \end{algorithmic}
  \end{algorithm}
  
  To reduce the number of calculations when $n=m$, we use the fact that $\Gamma\colon V_n\otimes V_m\to V_k$ is either symmetric or skew-symmetric, depending only on $2n-k\mod{4}$. Therefore we restrict our construction of $\Gamma$ to the corresponding subspace $S^2(V_n)$ or $\Lambda^2(V_n)$ and extend to $V_n\otimes V_n$ by symmetry or antisymmetry, see Steps \ref{alg_step:symmetric} -\ref{alg_step:bilinear}.

 A further optimisation of our algorithm is our implementation of step \ref{alg_step:gamma_inverse}.
 In order to calculate $[\mathcal{B}]^{-1}$, we order the basis $\mathcal{B}$ by decreasing weights to obtain a block matrix and then calculate the inverses of each block.
 With this process we invert $n+m+1$ blocks of size at most $(m+1)\times(m+1)$ instead of inverting the full matrix, which has size $(n+1)(m+1)\times (n+1)(m+1)$.
 This gives a significant improvement in complexity.
Indeed, when using the Gauss-Jordan algorithm, which has complexity $O(N^3)$ on the size of the matrix, our optimisation leads to a complexity of $O(m^3n)$ in contrast to the direct calculation that has complexity $O(m^3n^3)$.
 
\subsection{Algorithm to calculate $\Der(\g)$}\label{sec:codigo_derivaciones}
The input for this algorithm is the set of adjoint matrices of the basis of $\g$ and its output is a matrix containing a basis of $\Der(\g)$. 

\medskip
Let $e_1,\ldots,e_n$ be a basis for $\g$ and $\ad_{e_i}$ the associated adjoint to $e_i$. A linear map $D\colon\g\to\g$ is a derivation if and only if $D = (d_{r,s})_{1\leq r,s\leq n}$ satisfies 
\begin{align*}
  D\circ \ad_{e_i} (e_j) &= [D(e_i),e_j] + \ad_{e_i}\circ D (e_j)\\
  &= \sum_{k=1}^nd_{k,i}[e_k,e_j] + \ad_{e_i}\circ D (e_j)\\
  &= \sum_{k=1}^nd_{k,i}\ad_{e_k}(e_j) + \ad_{e_i}\circ D (e_j)
  \end{align*} for every $1\leq i,j\leq n$. 
This identity is equivalent to equality of linear operators
\[
  D\circ \ad_{e_i} = \sum_{k=1}^nd_{k,i}\ad_{e_k} + \ad_{e_i}\circ D,
\] for all $1\leq i \leq n$. Using the adjoint matrices as an input, the above identities allow us to define a system of linear equations, where $D$ is the unknown. We then use the isomorphism $\C^n\times \C^n \simeq \C^{n^2}$ to obtain a basis for the space of solutions of $D$. That is, we obtain a set of generators of the vector space $\Der(\g)$ and as a consequence, its dimension.

\subsection{Algorithm to verify if a vector space is a Lie algebra.}\label{sec:codigo_es_alg_Lie}
Given candidate structure constants as input, this algorithm verifies if these are the structure constants of a Lie algebra. That is, it checks the Jacobi identity.

Using candidate structure constants $[e_i,e_j]= \sum_k{c^i_{jk}}e_k$ for $1\leq i < j\leq n$, we first construct associated adjoint matrices $\ad_{e_i}$ using the antisymmetry of a Lie bracket. 
In our code we verify if the identity 
\[
\ad_{[e_i,e_j]} = \ad_{e_i}\circ \ad_{e_j} - \ad_{e_j}\circ \ad_{e_i}    
\] holds for every $1\leq i < j \leq n$.

\subsection{The sympathetic Lie algebra of dimension 25}

In this section we present how our algorithms construct, in practice, the isomorphisms given by the Weyl and Clebsh-Gordan theorems and Schur's lemma. We work with a concrete example: Benayadi's $25$-dimensional sympathetic Lie algebra.

\medskip
Consider the vector space $V_6$, generated by $\{e_0,e_1,\ldots , e_6\},$ and its corresponding representation $\rho_6:\s_2(\C) \to \gl(V_6)$, defined in \Cref{prop:explicit_irreps}. By \Cref{cor:Clebsh-Gordan_sym_ext} we have a decomposition
\[
 \Lambda^2 (V_6) \simeq  V_{10}\oplus V_6 \oplus V_2.  \eqno{(1)}
\]
In order to define an $\s_2(\C)$-equivariant morphism between $\Lambda^2 (V_6)$ and $V_6$, we use that
\[
\begin{aligned}
\Hom_{\s_2(\C)}(\Lambda^2 (V_6), V_6)  & \simeq \Hom_{\s_2(\C)} (V_{10}\oplus V_6 \oplus V_2 ,V_6)  \\
& \simeq  \Hom_{\s_2(\C)} (V_{10} ,V_6)\oplus  \Hom_{\s_2(\C)} (V_6,V_6) \\
& \qquad \oplus  \Hom_{\s_2(\C)} (V_2 ,V_6)
\end{aligned}
\]

From Schur's Lemma, it follows that 
\[
  \Hom_{\s_2(\C)} (V_{10} ,V_6)\equiv \Hom_{\s_2(\C)} (V_2 ,V_6) \equiv 0,
\] and hence
\[
\Hom_{\s_2(\C)}(\Lambda^2 (V_6), V_6)  \simeq  \Hom_{\s_2(\C)} (V_6,V_6) \eqno{(2)}.
\]

In order to construct this morphism, it is necessary to obtain the explicit decomposition from eq. (1), which can be achieved using standard representation theory: find a basis of highest weight vectors of $\Ker(\hat{\rho}_6 (E))$, where $\hat{\rho}_6=\rho_6\otimes \rho_6$.

\medskip
Define $e_i \wedge e_j := e_i \otimes e_j - e_j \otimes e_i$. A straightforward calculation shows
\begin{Lemma}
$\dim \Ker(\hat{\rho}_6(E))= 3$, and,
$\Ker(\hat{\rho}_6(E)) = \langle  e_0\wedge e_1, 
- e_0\wedge e_3 +2 e_1\wedge e_2,
3 e_0\wedge e_5 - 5 e_1 \wedge e_4 + 6  e_2 \wedge e_3 \rangle$. 
\end{Lemma}

 \begin{Cor}
 Let $u_{10} :=e_0\wedge e_1$, 
 $u_6:= - e_0\wedge e_3 +2 e_1\wedge e_2$ y
 $u_2:=3 e_0\wedge e_5 - 5 e_1 \wedge e_4 +6 e_2 \wedge e_3$. Then
\[
V_i \simeq \langle u_i, \hat{\rho}_6(F)(u_i), \ldots, \hat{\rho}_6(F)^i (u_i)  \rangle
\]
for $i=2,6,10$.
 \end{Cor}

\begin{Cor}\label{cor:morphism}
The morphism in Eq. (2) is given by:
\begin{itemize}
\item $\hat{\rho}_6(F)^j (u_{10}) \mapsto 0$, for $j=0,\ldots, 10$,
\item $\hat{\rho}_6(F)^j (u_6) \mapsto e_j$ for $j=0,\ldots, 6$,
\item $\hat{\rho}_6(F)^j (u_2) \mapsto 0$  for $j= 0,1,2$.
\end{itemize}
\end{Cor}

\Cref{cor:morphism} allows us to define a skew-symmetric $\s_2(\C)$-equivariant morphism 
$\Gamma_{6,6}^6 : V_6 \times V_6 \to V_6$. To construct $\Gamma_{6,6}^6 : V_6 \times V_6 \to V_6$ we solve the linear system of equations originating from \Cref{cor:Clebsh-Gordan_sym_ext}. For example, it is easy to verify the identities
\[
u_6= - e_0\wedge e_3 + 2 e_1 \wedge e_2, \quad\text{ and } \quad
\hat{\rho}_6(F)^2(u_{10}) =  e_0\wedge e_3 + e_1\wedge e_2.
\]
Thus
\[
- e_0\wedge e_3 + 2 e_1 \wedge e_2 \mapsto e_0, \quad
 e_0\wedge e_3 + e_1\wedge e_2 \mapsto 0
\]

In terms of $\Gamma_{6,6}^6$ this is equivalent to:
\[
\begin{aligned}
- \Gamma_{6,6}^6(e_0, e_3) + 2 \Gamma_{6,6}^6(e_1,e_2)  & = e_0 \\
 \Gamma_{6,6}^6(e_0,e_3) +  \Gamma_{6,6}^6(e_1,e_2) & = 0.
\end{aligned}
\]
From these equations one can easily obtain $\Gamma_{6,6}^6(e_0,e_3)$ y $\Gamma_{6,6}^6(e_1,e_2)$.
Moreover,
\begin{Lemma}\label{lem:V_6_x_V_6_to_V_6}
Up to scalar multiples, the only skew-symmetric $\s_2(\C)$-equivariant morphism $\Gamma_{6,6}^6: V_6 \times V_6 \to V_6$ is given by:
\[
\begin{matrix}
\Gamma_{6,6}^6 & e_0 & e_1 & e_2 & e_3 & e_4 & e_5 & e_6 \\
&&&&&&& \\
e_0 &  0  &  0  &  0  & - e_0 & -2 e_1 & -2 e_2 & -  e_3 \\
&&&&&&& \\
e_1 & &  0  &  e_0 & e_1 &  0  & -  e_3 &  -  e_4 \\
&&&&&&& \\
e_2 & & &  0  &  e_2 &  e_3 &  0  & -  e_5 \\
&&&&&&&\\
e_3 & & & &  0  &  e_4 & e_5 & -  e_6 \\
&&&&&&&\\
e_4 & & & & &  0  & 2 e_6 &  0  \\
&&&&&&&\\
e_5 & & & & & &  0  &  0 \\
&&&&&&&\\
e_6  & & & & & & & 0 \\
\end{matrix}
\]
\end{Lemma}

\medskip
Analogously, we conclude that $\Hom_{\s_2(\C)} (\Lambda^2V_6, V_2) \simeq
\Hom_{\s_2(\C)} (V_2, V_2)$ and obtain:

\begin{Lemma}\label{lem:V_6_x_V_6_to_V_2}
Up to scalar multiples, the only bilinear, skew-symmetric $\s_2(\C)$-equivariant morphism 
$\Gamma_{6,6}^2: V_6 \times V_6 \to V_2:=\langle f_0, f_1, f_2\rangle$ is given by:
$$
\begin{matrix}
\Gamma_{6,6}^2 & e_0 & e_1 & e_2 & e_3 & e_4 & e_5 & e_6 \\
&&&&&&& \\
e_0 &  0  &  0  &  0  &  0  &  0  &  f_0 & 3 f_1 \\
&&&&&&& \\
e_1 & &  0  &  0  &  0  & - f_0 & - 2 f_1 & 3 f_2 \\
&&&&&&& \\
e_2 & & &  0  &  f_0 &  f_1 & -5 f_2 &  0  \\
&&&&&&&\\
e_3 & & & &  0  & 6 f_2 &  0  &  0  \\
&&&&&&&\\
e_4 & & & & &  0  &  0  &  0  \\
&&&&&&&\\
e_5 & & & & & &  0  &  0 \\
&&&&&&&\\
e_6  & & & & & & &  0  \\
\end{matrix}
$$
\end{Lemma}

\begin{Lemma}\label{lem:V_6_x_V_4_to_V_2}
Up to scalar multiples, there is a unique bilinear $\s_2(\C)$-equivariant morphism
$\Gamma_{6,4}^2: V_6 \times V_4 \to V_2=\langle g_0, g_1, g_2 \rangle$ and it is given by:
$$
\begin{matrix}
\Gamma_{6,4}^2 & f_0 & f_1 & f_2 & f_3 & f_4  \\
&&&&&&& \\
e_0 &  0  &  0  &  0  &  0  & g_0 \\
&&&&&&& \\
e_1 &  0  &  0  &  0  & -  g_0 &  g_1 \\
&&&&&&& \\
e_2 &  0  &  0  &  g_0 & - 2 g_1 &  g_2  \\
&&&&&&&\\
e_3 &  0  & - g_0 &  3g_1& - 3 g_2 &  0  \\
&&&&&&&\\
e_4 &  g_0 & - 4 g_1 &  6 g_2  &  0  &  0   \\
&&&&&&&\\
e_5 & 5 g_1 & -10 g_2  &  0  &  0  &  0  \\
&&&&&&&\\
e_6  & 15 g_2 &  0  &  0  &  0  &  0  \\
\end{matrix}
$$
\end{Lemma}

\medskip

One can then extend $\Gamma_{6,4}^2$ to a unique bilinear skew-symmetric and $\s_2(\C)$-equivariant morphism  
\[\Gamma_{(6,4)}^2: (V_6 \oplus V_4) \times (V_6\oplus V_4) \to V_2,\] by defining
\[
\Gamma_{(6,4)}^2\vert_{V_6 \times V_6} \equiv 0, \quad \Gamma_{(6,4)}^2\vert_{V_4 \times V_4} \equiv 0,\]
and
\[
\Gamma_{(6,4)}^2\vert_{V_6 \times V_4} = \Gamma_{6,4}^2 \quad
\Gamma_{(6,4)}^2\vert_{V_4 \times V_6} = - \Gamma_{6,4}^2,
\]
 where 
\[
\Gamma_{(6,4)}^2(u,w):= - \Gamma_{6,4}^2 (w,u),
\] for $u\in V_4$, $w\in V_6$.

\begin{Lemma}\label{lem:V_4_x_V_4_to_V_2}
Up to scalar multiples, the only skew-symmetric $\s_2(\C)$-equivariant morphism 
$\Gamma_{4,4}^2: V_4 \times V_4 \to V_2$ is given by:
$$
\begin{matrix}
\Gamma_{4,4}^2 & e_0 & e_1 & e_2 & e_3 & e_4  \\
&&&&& \\
e_0 &  0  &  0  &  0  & - f_0 & -2 f_1 \\
&&&&& \\
e_1 & &  0  &  f_0 &  f_1 & -2 f_2  \\
&&&&& \\
e_2 & & &  0  & 3 f_2 &  0   \\
&&&&&\\
e_3 & & & &  0  &  0  \\
&&&&&\\
e_4 & & & & &  0   \\
\end{matrix}
$$
\end{Lemma}

We modified the notation used in Benayadi's original construction, \cite{B}. Here $V_6=D(3)$, $V_4=D(2)$ and $V_2=D(1)$.

\medskip
As a vector space, $\g=\s_2(\C) \oplus V_6 \oplus V_4 \oplus V_6 \oplus V_2$ and so $\dim(\g) = 25$. The algebra structure of $\g$ is then defined as:

$$
\begin{matrix}
    [,]_{\g}   & \s_2(\C)         & & V_{6_1}      &       & V_4                 & & V_{6_2}    & V_2    \\
    &&&&&&&& \\
 \s_2(\C)  &  [,]_{\s_2(\C)}  & & \rho_6        & & \rho_4             & & \rho_6  & \rho_2   \\
     &&&&&&&& \\
V_{6_1} &                                &&  \Gamma_{6_1,6_1}^{6_2} && \Gamma_{6_1,4}^2  && \Gamma_{6_1,6_2}^2 &  0  \\
 &&&&&&&& \\
V_4 &                               &&                              && \Gamma_{4,4}^2  &&  0  &  0  \\
  &&&&&&&& \\
V_{6_2} &                             &&                                &&                             &&  0  &  0  \\
&&&&&&&& \\
V_2  &                             &&                               &&                             &&     &  0 
 \end{matrix}
 $$

Here $\Gamma_{a,b}^c : V_a \times V_b \to V_c$ denotes the unique, up to scalar multiples, bilinear $\s_2(\C)$-equivariant morphism obtained from Schur's Lemma. For example, $\Gamma_{6,6}^6$ is defined by \Cref{lem:V_6_x_V_6_to_V_6} and we are using the notation $V_{6_2}$ to indicate that projection is onto the second copy of $V_6$.

\medskip
Using the algorithm \Cref{sec:codigo_generar_V_k} we are able to explicitly calculate basis for $\Lambda^2 V_6$ and $\Lambda^2 V_4$ and therefore establish \Cref{lem:V_6_x_V_6_to_V_6}, \Cref{lem:V_6_x_V_4_to_V_2}, \Cref{lem:V_6_x_V_4_to_V_2}.

\medskip
Given the structure constants that we obtained, the algorithm in \Cref{sec:codigo_es_alg_Lie}  then verifies that $\g$ is indeed a $25$-dimensional Lie algebra.

\medskip
Finally, we use the algorithm \ref{sec:codigo_derivaciones} to explicitly calculate $\Der(\g)$ and obtain its dimension $\dim_{\C} \Der(\g)=25$. In this case we know $0= Z(\g) = \Ker(\ad_{\g})$ and so 
$\dim_{\C} \ad(\g) = 25$. Since $\dim_{\C} \Der(\g)=\dim_{\C} \ad(\g)$, we can conclude:

\begin{Prop} Let $\g$ be as in the previous construction. Then,
\begin{enumerate}
\item $Z(\g)=\{ 0\}$,
\item $\g = [\g,\g]$,
\item $\Der(\g) = \ad(\g)$.
\end{enumerate}
That is, $\g$ is a sympathetic Lie algebra.
\end{Prop}
\begin{proof}
Notice that by construction (2) $[\s_2(\C), V_i] = V_i$ and therefore $\g = [\g,\g]$.
\end{proof}

\bigskip
\subsection*{Acknowledgements} 
ALGP is a member of the Centre for Topological Data Analysis funded by the EPSRC grant “New
Approaches to Data Science: Application Driven Topological Data Analysis” number EP/R018472/1. 
For the  purpose of Open Access, the authors have applied a CC BY public copyright licence to any  Author Accepted Manuscript (AAM) version arising from this submission.
GS would like to acknowledge the partial support received
by CONACyT grant A-S1-45886 and by PRODEP grant UASLP-CA-228.

\bibliography{references}\bibliographystyle{abbrvurl}
\end{document}